\documentclass[12pt,twoside]{amsart}
\usepackage{amsfonts}
\usepackage{fancyhdr}
\usepackage{amsmath} 
\usepackage{hyperref}
\usepackage{lscape}
\usepackage{color}

\usepackage{amssymb}
\theoremstyle{plain}

\makeatletter
\@addtoreset{equation}{section}

\makeatother
\newtheorem{theo}{Theorem}[section]
\newtheorem{pro}[theo]{Proposition}
\newtheorem{lemm}[theo]{Lemma}
\newtheorem{cor}[theo]{Corollary}

\newtheorem{rem}[theo]{Remark}
\newtheorem{deff}[theo]{Definition}

\newcommand{\avert}[1]{-\hskip-0.41cm\int_{#1}}

\numberwithin{equation}{section}

\begin{document} 
\allowdisplaybreaks

\title[$L^p$ estimates for magnetic Schr\"odinger operators]{Maximal inequalities and Riesz transform estimates on $L^p$ spaces for magnetic Schr\"odinger operators I}
\author{Besma Ben Ali}
\address{ B. Ben Ali
\\
Université de Paris-Sud, UMR du  CNRS  8628
\\
91405 Orsay Cedex, France} \email{besma.ben-ali@math.u-psud.fr}

\begin{abstract}
The paper concerns the  magnetic Schr\"odinger operator $H(\textbf{a},V)=\sum_{j=1}^{n}(\frac{1}{i}\frac{\partial}{\partial
x_{j}}-a_{j})^{2}+V $ on $\mathbb{R}^n$. Under certain conditions, given in terms of the reverse H\"older inequality on the magnetic field and the electric potential, we prove some $L^p$ estimates on the Riesz transforms of $H$ and we establish some related maximal inequalities.

\end{abstract}
\date{}
\maketitle

\begin{quote}
{\tableofcontents}
\end{quote}
\vline

\section{Introduction}
Consider the Schr\"odinger operator with magnetic field
\begin{equation}\label{def}H(\textbf{a},V)=\sum_{j=1}^{n}(\frac{1}{i}\frac{\partial}{\partial
x_{j}}-a_{j})^{2}+V \, \textrm{in}\, \, \mathbb{R}^{n},\qquad
n\geq 2,
 \end{equation}
  where $\textbf{a}=(a_{1},a_{2},\cdots,a_{n}):
\mathbb{R}^{n}\rightarrow \mathbb{R}^{n}$ is the magnetic potential and $V:\mathbb{R}^{n}\rightarrow \mathbb{R}$ is the electric potential.
Let \begin{equation}B(x) = curl\, \textbf{a}(x)= (b_{jk}(x))_{1\leq j,k\leq n}\end{equation}
 be the magnetic field generated by $\textbf{a}$, where
 \begin{equation}
 b_{jk} = \frac{\partial a_{j}}{\partial x_{k}}- \frac{\partial a_{k}}{\partial x_{j}}.
 \end{equation} 
 We will assume that $\textbf{a}\in L^{2}_{loc}(\mathbb{R}^n)^n$ and $V\in L^{1}_{loc}(\mathbb{R}^n), \,\, V\geq 0$.
Let
 \begin{equation} L_{j}=
\frac{1}{i}\frac{\partial}{\partial x_{j}}-a_{j} \qquad
\textrm{for}\qquad 1\leq j\leq n, 
 \end{equation} 
 Set $L=(L_{1},\ldots, L_{n})$ and $
|L u(x)|=\big{(} \sum_{j=1}^{n} |L_{j}u(x)|^{2}\big{)}^{1/2}.$
 
Note that $ L^{\star}_{j}=L_{j}$ for all $1\leq j\leq n,$ and let
$$L^{\star}=(L^{\star}_{1},\ldots, L^{\star}_{n})^{T}.$$
We define the form $\mathcal{Q}$ by
\begin{equation}\mathcal{Q}(u,v)= \sum_{k=1}^{n} \int_{\mathbb{R}^n}
L_{k} u . \overline{L_{k} v} dx+\int_{\mathbb{R}^n}V  u . \bar{ v} dx ,
\end{equation} 
with domain $\mathcal{D}(\mathcal{Q})=\mathcal{V}\times \mathcal{V}$ where $$\mathcal{V} = \{u\in
L^{2}, L_{k}u \in L^{2} \,\, \textrm{for } k=1,\ldots,n \,\,\textrm{and}\, \,  \sqrt{V} u \in L^{2}\}.$$
 We denote $H(\textbf{a},V)=H$, the self-adjoint operator on $L^2(\mathbb{R}^n)$ associated to this symmetric and closed form.
 
  The domain of $H $ is given by:
 $$\mathcal{D}(H)=\{u \in \mathcal{D}(\mathcal{Q}), \exists v\in L^{2} \,\, \textrm{so that}\,\, \mathcal{Q}(u,\phi)=\int_{\mathbb{R}^n} v\bar{\phi} dx, \, \forall \phi \in \mathcal{D}(\mathcal {Q}) \}.$$

 The operators $L_{j} H(\textbf{a},V)^{-1/2}$ are called the Riesz transforms associated with $H(\textbf{a},V)$.
 We know that
 \begin{equation}\label{LL2} 
 \sum^{n}_{j=1}\| L_{j} u\|^{2}_{2}+ \|V^{1/2}u\|^{2}_{2} = \| H(\textbf{a},V)^{1/2} u\|^{2}_{2}, \qquad \forall u \in \mathcal{D}(\mathcal{Q}) =\mathcal{D}(H(\textbf{a},V)^{1/2}).
 \end{equation}
 Hence, the operators $L_{j} H(\textbf{a},V)^{-1/2}$ are bounded on $L^{2}(\mathbb{R}^{n})$, for all $j=1,\ldots,n$.

 The aim of this paper is to establish the $L^p$  boundedness of the operators $L_{j} H(\textbf{a},V)^{-1/2}$ and $V^{\frac{1}{2}}H(\textbf{a},V)^{-\frac{1}{2}}$. 
  In the presence of the magnetic field, the only known result is that these operators are of weak type (1.1) and hence, by interpolation, are  $L ^p$ bounded for all $ 1 <p \leq  2$.
This result was proved by Sikora using  the finite speed propagation property \cite{Sik}. Independantly, Duong, Ouhabaz and Yan have proved the same result using another method.

Many authors have been interested in the study of the Riesz transforms of $H(\textbf{a},V)$ in the case when the magnetic potential $\textbf{a}$ is absent, i.e $LH(\textbf{a},V)^{-\frac{1}{2}}=\nabla(-\Delta+V)^{-\frac{1}{2}}$. We mention the works of Helffer-Nourrigat \cite{HNW}, Guibourg\cite{Gui2} and Zhong \cite{Z}, in which they considered the case of polynomial  potentials. A generalization of their results was given by Shen \cite{Sh1}, he proved the $L^p$ boundedness of Riesz transforms of Schr\"odinger operators with electric potential contained in certain reverse H\"older classes. Auscher and I improved this result in \cite{AB}, using a different approach based on local estimates. Note that this approach can be extended to more general spaces for instance some Riemannian manifolds and Lie groups( see \cite{BB}).
The main purpose of this work is to find some sufficient conditions on the electric potential and the magnetic field, for which the Riesz transforms of $H(\textbf{a},V)$ are $L^p$ bounded for a range $p>2$. Many arguments follow those of \cite{AB}, the contribution of the magnetic field will be controlled by introducing an auxiliary function $m(.,|B|)$.

 Note that, because of the gauge invariance of the operator $H(\textbf{a},V)$ and the nature of the $L^{p}$ estimates, any quantitative condition should be imposed on magnetic field $B$ , not directly on $a$.

 This article also aims to establish some maximal inequalities related to the $L^p$ behaviour of $L_{j}L_{k} H(\textbf{a},V)^{-1}$, $V H(\textbf{a},V)^{-1}$ and other operators called the second order Riesz transforms. The only known result for a range $p>2$ is given by Shen in \cite{Sh4}. He generalized the $L^2$ estimate proved by Guibourg in \cite{Gui1} for polynomial potentials. Estimates on these operators are of great interest in the study of spectral theory of $H(\textbf{a},V)$. 
 In this paper our assumptions on potentials will be given in terms of reverse H\"older inequality. Let recall the definition of these weight classes:
  
 \begin{deff}
 Let $\omega \in L^q_{loc}(\mathbb{R}^n)$,  $\omega>0$ almost everywhere, $\omega\in RH_{q}$, $1<q\le\infty$, the class of the reverse H\"older weights with exponent $q$, if there exists a constant $C$ such that  for any cube $Q$ of $\mathbb{R}^n$,

 \begin{equation}\label{2RHclass}
 \Big{(} \avert{Q} {\omega^{q}}(x)dx \Big{)}^{1/q}\leq C\Big{(}\avert{Q} \omega (x)dx \Big{)}.
 \end{equation}
If $q=\infty$, then the left hand side is the essential supremum on $Q$. 
The smallest $C$ is called the $RH_{q}$ constant of $\omega$.
 \end{deff}
\textit{A note about notations:}
Throughout this paper we will use the following notation $\avert {Q}\omega=\frac{1}{|Q|}\int_{Q}\omega.$
$C$ and $c$ denote constants. As usual, $\lambda Q$ is the cube co-centered with $Q$ with sidelength $\lambda$ times that of $Q$.

We give the definition of an auxiliary function introduced by Shen in \cite{Sh1}
 \begin{deff}
 
 Let $\omega\in L^{1}_{loc}(\mathbb{R}^{n})$, $\omega\geq 0$, for $x\in \mathbb{R}^n$, the function $m(x,\omega)$ is defined by:
 \begin{equation} \frac{1}{m(x,\omega)}= \sup \left\{ r>0: \frac{r^2}{| Q(x,r)|} \int_{Q(x,r)} \omega(y) dy \leq 1\right\}.
 \end{equation}
 \end{deff}
  
 
 We now state our main result :
 \begin{theo}\label{th:1a} Let $\textbf{a}\in L^{2}_{loc}(\mathbb{R}^{n})^{n}$. Also assume the following conditions \begin{equation}\label{eq:shen}
 \left\{\begin{array}{ll}
 | B |  \in RH_{n/2}\\ | \nabla B |\leq c\, m(.,| B|)^{3},
  \end{array}
  \right.
 \end{equation}
 where $| B |=\sum_{j,k} | b_{jk} |$ and $\nabla=(\frac{\partial}{\partial x_{1}},\ldots,\frac{\partial}{\partial x_{n}})$ . Then, for all $1<p<\infty$, there exists a constant $C_{p}>0$, such that
 \begin{equation}
 \label{eq:Lp}  \| LH(\textbf{a},0)^{-1/2}(f)\|_{p}\leq C_{p} \| f\|_{p},
 \end{equation}
 for any $f\in C_{0}^{\infty}(\mathbb{R}^{n})$,\\ and
 $$|\{x\in \mathbb{R}^n\, ; \,  |L f(x)| > \alpha\}| \le \frac{C_{1}}{\alpha} \|H(\textbf{a},0)^{1/2}f\|_{1}.$$
 for $\alpha>0$ and all $f\in C_{0}^{\infty}(\mathbb{R}^{n})$ if $p=1$.

 \end{theo}

The conditions \eqref{eq:shen}, which are dilation invariant, are used by Shen in \cite{Sh4} to study the operators $ L_{j}L_{k} H (a, V) ^{-1} $.
Note that these conditions mean that the value of $|B|$ do not fluctuate too much on the average and $|\nabla B|$ is uniformly bounded in the scale ${m(x, |B|)}^{-1}$. It is clear that the hypothesis of Theorem \ref{th:1a} is satisfied if the magnetic potentials $a_{j}(x)$ are polynomials. 
  	
Once the estimates for the pure magnetic Schr\"odinger operator  $ H (\textbf{a}, 0) $ is established, we will proceed onto the second part of our work. We then add the positive electric potential  $ V \in RH_q $, with $ q> $ 1, while keeping the same conditions on $B $ and get the following theorem:
  \begin{theo}\label{th:2a} Let $\textbf{a}\in L^{2}_{loc}(\mathbb{R}^{n})^{n}$, $V\in RH_{q}$, $1<q\leq \infty$. Also assume that the magnetic field $B$ satisfies the conditions \eqref{eq:shen}. 
 
  Then, there exists an $ \epsilon>0$ depending on the reverse H\"older constant $RH_q$ of $V$, such that, for every $1<p< sup(2q,q^{\star})+\epsilon$, there exists a constant $C_{p}>0$, such that
 \begin{equation}\label{Lp}\| LH(\textbf{a},V)^{-1/2}(f)\|_{p}\leq C_{p} \| f\|_{p},
 \end{equation}
 for any $f\in C_{0}^{\infty}(\mathbb{R}^{n}).$
 Here, $q^*= qn/(n-q)$ is the Sobolev exponent of $q$ if $q<n$, and $q^{\star}=\infty$ if $q\geq n$.

 \end{theo}

Taking into account the conditions on the electric potential, and persuing step-by-step the proof of Theorem \ref{th:1a}, we get the following result
 \begin{theo}\label{Lpshen} 
Let $\textbf{a}\in L^{2}_{loc}(\mathbb{R}^{n})^{n}$, $V\in L^{1}_{loc}(\mathbb{R}^{n})$ and $V\geq 0$ a.e on  $\mathbb{R}^{n}$. Also assume that there exist two positive constants $c>0$ and $C>0$ such that:
 \begin{equation}\label{eq:Shen1}
 \left\{\begin{array}{ll}
 | B | + V \in RH_{n/2},\\V\leq C\, {m(.,|B|+V)}^{2},\\ | \nabla B |\leq c \,m(.,| B|+V)^{3}.
  \end{array}
  \right.
 \end{equation}
Then \eqref{Lp} is satisfied for all $1<p<\infty$. 
 \end{theo}
 	
The following three results will be useful to prove Theorem \ref{th:1a} and Theorem \ref{th:2a}. The first describes reverse inequalities of \eqref{Lp}. 
  
   \begin{theo}\label{th:2b}Let $V\in A_{\infty}$ or $V=0$, $\textbf{a}\in
L^{2}_{loc}(\mathbb{R}^{n})^{n}$ and $| B |  \in RH_{n/2}$.

 Then, for all $1\leq p<\infty$, there exists a constant $C_{p}>0$ depending only on the $RH_{\frac{n}{2}}$ constant of  $|B|$, such that \begin{equation}\label{eq:rev Lp'}\|
H(\textbf{a},V)^{1/2}(f)\|_{p}\leq C_{p}\{\| Lf\|_{p}+ \| |B|^{1/2}\,f\|_{p} +\| V^{1/2}\,f\|_{p} \}
\end{equation}  for any $f\in C_{0}^{\infty}(\mathbb{R}^{n})\,\textrm{if}\,\, p>1,$\\ and
 \begin{equation}\label{eq:rev L1'}|\{x\in \mathbb{R}^n\, ; \,  |H(\textbf{a},V)^{1/2} f(x)| > \alpha\}| \le \frac{C_1}{\alpha} \int |L f| + | |B|^{1/2}\,f|+|V^{1/2}\,f|,
 \end{equation}
for all $\alpha>0$ and  $f\in C_{0}^{\infty}(\mathbb{R}^{n})\,\textrm{if}\,\, p=1.$

 \end{theo} 
\begin{rem}
\begin{enumerate}
\item
Under assumptions \eqref{eq:shen},  we can replace $\| |B|^{1/2}\,f\|_{p} $ by $\|m(.,|B|)\,f\|_{p}$ in \eqref{eq:rev Lp'} and \eqref{eq:rev L1'}.
\item Under the conditions \eqref{eq:Shen1}, we can replace the term $\| |B|^{1/2}\,f\|_{p}+\| V^{1/2}\,f\|_{p} $ by $\|m(.,|B|+V)\,f\|_{p}$.
\end{enumerate}
	
Note that introducing \eqref{eq:shen} and \eqref{eq:Shen1} makes the proof of Theorem \ref{th:2b}, using the same strategy as before, easier.
\end{rem}
  The result concerns some new inequalities:   
   \begin{theo}\label{2th:VH}  Let $\textbf{a}\in  L^{2}_{loc}(\mathbb{R}^{n})^{n}$ and $ V
\in RH_{q},\quad 1< q\leq +\infty$. Then, there exists $ \epsilon >0$, depending only on the $RH_q$ constant of  $V$, such that $V\,H(\textbf{a},V)^{-1}$ and $H(\textbf{a},0)H(\textbf{a},V)^{-1}$ are $L^p$ bounded for all $1 \leq p
< q+\epsilon$.
\end{theo}
It follows by complex interpolation ( see \cite{AB} for more details):
\begin{cor}\label{2th:VH2}  Let $\textbf{a}\in  L^{2}_{loc}(\mathbb{R}^{n})^{n}$ and $ V
\in RH_{q}, \, \, 1<q\leq +\infty$. Then, there exists an $ \epsilon >0$, depending only on the $RH_q$ constant of  $V$, such that, the operators $V^{1/2}\,H(\textbf{a},V)^{-1/2}$ and $H(\textbf{a},0)^{1/2}\, H(\textbf{a},V)^{-1/2}$ are $L^p$ bounded for all $1< p< 2q+\epsilon$.
\end{cor}

We would give an alternative proof of the following theorem proved by Shen in\cite{Sh4}:
\begin{theo}\label{Maxshen}Under the conditions of Theorem \ref{Lpshen}, for all $s=1,\ldots.,n$ and $ k=1,\ldots,n$, the operators $L_{s}L_{k}H(\textbf{a},V)^{-1}$ are $L^p$ bounded for any $1<p< \infty$\footnote{Shen also proved a weak (1,1) type estimate for these operators.}.

 \end{theo}
 Note that with more general conditions on the electric potential, we have the following new result:
 \begin{theo}\label{Maxbes} Under the conditions of Theorem \ref{th:2a}, for all $s=1,\ldots.,n$ and $ k=1,\ldots,n$ , there exists an $\epsilon>0$ depending only on the $RH_q$ constant of  $V$, such that $L_{s}L_{k}H(\textbf{a},V)^{-1}$ are $L^p$ bounded for all $1<p< q+\epsilon$.
\end{theo}

 We mention without proof that our results admit local versions, replacing $V\in RH_{q}$ by  $V\in RH_{q,loc}$ which is defined by the same conditions on cubes with sides less than 1. Then we get the corresponding results and estimates  for $H+1$ instead of $H$.  The results on operator domains are valid under local assumptions.

Our arguments are based on local estimates. Our main  tools are

1) An improved Fefferman-Phong inequality for $A_{\infty}$ potentials.

2) Criteria for proving  $L^p$ boundedness of operators in absence of kernels. 

3)  Mean value inequalities for  nonnegative subharmonic functions against $A_{\infty}$ weights.

4) Complex  interpolation, together with  $L^p$ boundedness of imaginary powers 
of $H(\textbf{a},V)$ for $1<p<\infty$. 

5) A Calder\'on-Zygmund decomposition adapted to level sets of the maximal function of $|L f| + |V^{1/2}f| $. 

6) A gauge transform adapted to the reverse H\"older conditions on the potentials.

7)An auxiliary global weight controlling the contribution from the magnetic field.

8) Reverse H\"older inequalities involving $\L u$, $m(.,|B|) u$, $|B|^{1/2} u$ and $V^{1/2} u$ for weak solutions
of $H(\textbf{a},V) u=0$.

 The paper is organized as follows. In Section 2 we introduce some useful estimates. We state an improved Fefferman-Phong inequality and we establish an adapted gauge transform. Section 3 is devoted to the study of pure magnetic Schrodinger operator, first we establish some reverse estimates via a Calder\`on-Zygmund decomposition, then we prove the $L^p$ boundedness of Riesz transforms fo all $1<p<\infty$. In section 4 we consider the magnetic Schr\"odinger operator with electric potential, we study the $L^p$ behaviour of the first and the second order Riesz transforms.
\section{Preliminaries}

We begin by recalling some properties of the reverse H\"older classes.
 \begin{pro}\label{11.1}(Proposition 11.1 [AB]) Let $\omega$ be a nonnegative measurable function. Then the following are equivalent:
\begin{enumerate}
  \item $\omega\in A_{\infty}$.
  \item For all $s\in (0,1)$, $\omega^s\in B_{1/s}$.
  \item There exists $s\in (0,1)$, $\omega^s\in B_{1/s}$.
\end{enumerate}

\end{pro}
It is well known that if $\omega \in RH_q$ and $q<+\infty$, then $\omega \in RH_{p}$ for all $1<p<q$  and there exists an $\varepsilon >0$ such that $\omega\in RH_{q+\varepsilon}$.
 We also know that $\omega \in A_{\infty}$ if and only if there exists $q>1$ such that $\omega\in RH_{q}$. Here $A_{\infty}$ is the Muckenhoupt weight class, defined as the union of all $A_p$, $1\leq p<\infty.$
 If $\omega\in A_{\infty}$ then $\omega(x) dx$ is a doubling measure (see  \cite{St},chap V for more details).

We will also recall some important properties of the function $ m (., \omega) $:
 \begin{lemm}\label{th:mprop} Suppose $\omega\in RH_{n/2}$, then there exist $c>0$ and $C>0$ such that for all $x$ and $y$ in $\mathbb{R}^{n}$:
\begin{enumerate}
 \item  $ 0< m(x,\omega)< \infty$ for all $x\in \mathbb{R}^{n}$.
 \item Si $| x -y | < \frac{C}{m(x,\omega)}$, then $ m(x,\omega)\approx m(y,\omega).$
 \item $m(y,\omega)\leq C\{1+ | x-y| m(x,\omega)\}^{k_{0} }m(x,\omega).$
\item $m(y,\omega)\geq \frac{C\,m(x,\omega)}{\{1+ | x-y| m(x,\omega)\}^{k_{0}/(k_{0}+1)}}.$
for some $k_{0}$ depending on $\omega$.
 \end{enumerate}
 \end{lemm}

 	
We will see that if $ u $ is a weak solution of $ H (\textbf{a}, V) u = 0 $, it is easier to obtain reverse H\"older inequalities using terms $m(.,|B|)u$ and $Lu$ than is the case when we work with estimates of $|B|^{1/2}u$.

Fix an open set $\Omega$ and $f\in L^{\infty}_{comp}(\mathbb{R}^{n}) $, the space of compactly supported bounded functions on $\mathbb{R}^n$. By a weak solution of 
  \begin{equation}\label{12}
H(\textbf{a},V) u =f\,\, \textrm{in an open set}\,\Omega,
\end{equation}
we mean $u\in W(\Omega)$, with $$W(\Omega)=\{u\in L^1_{loc}(\Omega)\, ;\,     V^{1/2}u \, and \, L_{k} u \in L^2_{loc}(\Omega) \, \forall k=1,\ldots,n\}$$ and the equation \eqref{12} holds in the sense of distribution on $\Omega$. We note that if $u\in W(\Omega) $, then by Poincar\'e and the diamagnetic inequalities, $u \in L^2_{loc}(\Omega)$.

We will need the following tools:

 \begin{lemm}\textbf{Caccioppoli type inequality}\\
Let $u$ a weak solution of $H(\textbf{a},V)u=f$ in $2Q$, where $Q$ is a cube of $\mathbb{R}^n$ and $f\in L^{\infty}_{comp}(\mathbb{R}^{n})$. Then
 \begin{equation} \int_{Q} |Lu|^{2}+V|u|^{2}  \leq C\{ \int_{2Q} |f||u|  + \frac{1}{R^{2}} \int_{2Q}  |u|^{2} \}.
 \end{equation}

 \end{lemm}
 \begin{pro}\textbf{Diamagnetic inequality}\cite{LL}\\
For all $u\in W^{1,2}_{\textbf{a}}(\mathbb{R}^{n})$, with $$W^{1,2}_{\textbf{a}}(\mathbb{R}^{n})=\{u\in L^{2}(\mathbb{R}^n),\,\, L_{k} u \in L^{2}(\mathbb{R}^{n}),\,\, k=1\cdots , n\},$$
we have
\begin{equation}\label{2diam}|\nabla(|u|)|\leq| L(u)|.
\end{equation}

\end{pro}
\begin{pro}\textbf{Kato-Simon inequality}:
 \begin{equation}\label{eq:KS}| (H(\textbf{a},V)+\lambda)^{-1} f| \leq (-\Delta+ \lambda)^{-1}|f|; \qquad \forall f\in L^{2}(\mathbb{R}^n),\,\forall \lambda >0.
 \end{equation}
 \end{pro}
 \textbf{Fefferman-Phong inequalities}
The usual Fefferman-Phong inequalities are of the form:
 \begin{equation}\label{eq:feff}
  \int_{Q} | u |^{p}  \min\{\avert {Q}\omega, \frac{1}{R^p}\}\leq C\{\int_{Q} | Lu|^{p}  + \omega | u |^{p} \}.
 \end{equation}
 Shen proved in \cite{Sh3} the following  global version introducing the auxiliary weight function $m(., \omega)$ :
 \begin{lemm}Suppose $\textbf{a}\in L^{2}_{loc}(\mathbb{R}^{n})^{n}$. We also assume: 
\begin{equation}
 \left\{\begin{array}{ll}
 | B |+V  \in RH_{n/2}\\0\leq V\leq c\, m(.,| B|+V)^{2}\\ | \nabla B |\leq c'\, m(.,| B|+V)^{3}.
  \end{array}
  \right.
   \end{equation}
   Then, for all $u \in C^1(\mathbb{R}^{n})$,
\begin{equation}\|m(., |B|+V) u\|_{2}\leq C \big( \|Lu\|_{2}+\|V^{\frac{1}{2}}u\|_{2}\big).
\end{equation}
 \end{lemm}

In \cite{AB} we established an improved version for these inequalities in absence of the magnetic potential. We can extend this improvement to the magnetic Schr\"odinger operators:
\begin{lemm}\label{FP}\textbf{An improved Fefferman-Phong inequality} :\\ Let $\omega \in A_\infty$ and $1\le p<\infty$. Then there are constants $C>0$ and $\beta\in (0,1)$ depending only
on $p$, $n$ and the $A_\infty$ constant of $w$ such that for all cubes $Q$ $($with sidelength $R)$ and $u \in C^1(\mathbb{R}^{n})$, one has
\begin{equation}\label{1eq:FP}
\int_Q |L u|^p + \omega|u|^p \ge \frac{C m_{\beta}({R^p  \avert{Q} \omega})}{ R^{p}} \int_Q |u|^p
 \end{equation}
where $m_{\beta}(x)  = x$ for $x\le 1$ and $m_{\beta}(x) = x^\beta$ for $x\ge 1$. 
 
\end{lemm}
	
The proof is the same as that of Lemma 2.1 in \cite{AB}, combined with the diamagnetic inequality.

\begin{lemm}\label{2th:jauge} \textbf{Iwatsuka gauge transform}

 Let $\textbf{a}\in L^{2}_{loc}(\mathbb{R}^{n})^{n}$ and $Q$ a cube of $\mathbb{R}^n$. Suppose $B\in C^{1}(\mathbb{R}^{n}, M_{n}(\mathbb{R}))$. Then there exist $\textbf{h}\in C^{1}(Q,\mathbb{R}^{n})$ and a real function $\phi \in C^{2}(Q)$, 
 such that $curl \textbf{h}=B$ in $Q$ and
\begin{equation}\label{Gau}\textbf{h}=\textbf{a}- \nabla \phi,\qquad \textrm{in} \,\, Q,
\end{equation}
with
 
 \begin{equation}\label{2jaugebes}\big(\avert{Q} | \textbf{h}|^{n}\big)^{1/n}\leq  c\,R\big(\avert{Q} |B|^{\frac{n}{2}}\big)^{\frac{2}{n}},
\end{equation}
here $c$ depends only on $n$.
\end{lemm}
\begin{proof}We follow the proof of Lemma 2.4 in \cite{Sh5}, which uses the construction of Iwatsuka \cite{I}.

For $x,\,y\in Q$, let
$$g_{j}(x,y)= \sum^{n}_{k=1} (x_{k}-y_{k})\int^{1}_{0} b_{jk}(y+t(x-y))t dt$$
where $x=(x_{1},\ldots,x_{n}),\, y=(y_{1},\ldots,y_{n}).$\\
Let
$$h_{j}(x)= \avert {Q}g_{j}(x,y)dy, \, j=1,2,\ldots,n.$$
Then
$$|\textbf{h}(x)|=\bigg(\sum_{j} |h_{j}(x)|^{2}\bigg)^{1/2}\leq n^{\frac{n}{2}-1}\int_{Q} \frac{|B(y)|}{|x-y|^{n-1}}dy.$$

Now, we apply the  Hardy-Littlewood-Sobolev inequality (\cite{St},p.119) to get \eqref{2jaugebes}.
Hence \eqref{Gau} holds with

$$\phi(x)=\avert {Q}\bigg\{ \sum^{n}_{k=1}(x_{k}-y_{k})\int^{1}_{0}a_{k}(y+t(x-y))dt\bigg\}dy.$$
\end{proof}
\begin{cor} \label{jaugebes1} Let $\textbf{a}\in L^{2}_{loc}(\mathbb{R}^{n})^{n}$ and $Q$ a cube in $\mathbb{R}^n$. We assume that $curl \textbf{a}=B\in L^{n/2}_{loc}(\mathbb{R}^n, M_{n}(\mathbb{R}))$. 
 Then, there exist $\textbf{h}\in L^{n}(Q,\mathbb{R}^{n})$ and a real function $\phi \in H^{1}(Q)$, 
 such that $curl \textbf{h}=B$ a.e in $Q$ and
\begin{equation}\label{2phi}\textbf{h}=\textbf{a}- \nabla  \phi \qquad\textrm{a.e in}\, Q ,
\end{equation}
with
\begin{equation}\label{jaugebes1}\big(\avert{Q} | \textbf{h}|^{n}\big)^{1/n}\leq  c\,R\,\,\big(\avert{Q} |B|^{\frac{n}{2}}\big)^{\frac{2}{n}}.
\end{equation}

\end{cor}
\begin{proof}
Let $(\textbf{a}_{m})_{m\geq 0}$ be the sequence of $C^{1}$ functions obtained by convolution with $\textbf{a}$ and converge in $L^{2}_{loc}$ to $\textbf{a}$. Set $(B_{m})_{m\geq 0}$, $(\phi_{m})_{m\geq 0}$ and $(\textbf{h}_{m})_{m\geq 0}$ as the corresponding sequences of the Lemma \ref{2th:jauge}. Note that $(\textbf{h}_{m})_{m\geq 0}$ converges in $L^{n}(Q,\mathbb{R}^n)$. Let $\textbf{h}$ be this limit, it satisfies \eqref{2phi}.
Note also that $(B_{m})_{m\geq 0}$ converges to $B$ in $L^{n/2}_{loc}(Q, M_{n}(\mathbb{R}))$ and  $curl \textbf{h}= B$ holds always every where in $Q$, where curl is defined in the sens of distribution. 
 
We know that for all $ m\geq 0$, $$ (\avert{Q} | \textbf{h}_{m}|^{n})^{1/n}\leq  c\,R(\avert{Q} |B_{m}|^{\frac{n}{2}})^{\frac{2}{n}},$$
uniformly in $m$. Then 
applying the limit, we obtain
 $$(\avert{Q} | \textbf{h}|^{n})^{1/n}\leq  c\,R(\avert{Q} |B|^{\frac{n}{2}})^{\frac{2}{n}}.$$
 Hence inequality \eqref{2phi} follows easily.
 \end{proof}

 \section{Pure magnetic Schr\"odinger operator}
This section is devoted to establish $L^p$ estimates on Riesz transforms of $H(\textbf{a},0)$  as well as its converse. Since the electric potential is absent, we cannot follow the methods of \cite{AB}. An analogous approach based on local estimates requires different localization techniques. We also use a Calder\`on-Zygmund decomposition adapted to the presence of magnetic field via the gauge transform previously established.

  \subsection{Reverse estimates }

In the absence of electric potential, the theorem \ref{th:2b} is of the form:

  \begin{theo}\label{th:1b}Suppose $\textbf{a}\in
L^{2}_{loc}(\mathbb{R}^{n})^{n}$ and $| B |  \in RH_{n/2}$.

 Then, for all $1<p<\infty$, there exists $C_{p}>0$,
such that 
\begin{equation}\label{eq:rev Lp}
\|H(\textbf{a},0)^{1/2}f\|_{p}\leq C_{p}\big{(}\| Lf\|_{p}+ \| |B|^{1/2}\,f\|_{p}\big{)} 
\end{equation} 
 for all $f\in C_{0}^{\infty}(\mathbb{R}^{n})$.
There is a constant $C>0$ such that
\begin{equation}\label{eq:wt}|\{x\in \mathbb{R}^n\, ; \,  |H(\textbf{a},0)^{1/2} f(x)| > \alpha\}| \le \frac{C}{\alpha} \int |L f| + |B|^{1/2}|f|,
\end{equation}
 for $\alpha>0$ and for all $f\in C_{0}^{\infty}(\mathbb{R}^{n})$.

 \end{theo} 
 \begin{proof} We follow step by step the proof of the Theorem 1.2 of \cite{AB} once the appropriate Calder\'on-Zygmund decomposition \ref{th:CZ1} is established. We also use the fact that the time derivatives  of the kernel of semigroup $ e ^{-tH} $ satisfy Gaussian estimates (see \cite{CD}, \cite{Da}, \cite{G} and \cite{Ou} Or, theorem 6.17).
 \end{proof}
  Lets introduce the main technical lemma of this work, intersting in its own right:
 \begin{lemm}\label{th:CZ1} Let $1\leq p<n$ and $\alpha>0$. Suppose $\textbf{a}\in L^{2}_{loc}(\mathbb{R}^{n})^{n}$ and $|B| \in RH_{n/2}$. Let $f\in C^{\infty}_{0}(\mathbb{R} ^n)$ hence $$ \| L f
 \|_{p}+\| |B|^{1/2}f \|_{p} <\infty.$$ Then, one can find a collection of cubes $(Q_k)$ and functions $g$ and $b_k$  such that 
 \begin{equation}\label{eq:cza} f=g+\sum_{k} b_k,
 \end{equation}
 and the following properties hold:
 \begin{equation}\label{eq:czb}\| L g \|_{n}+\| |B|^{1/2} g \|_{n} \leq C\alpha^{1-\frac{p}{n} }\big{(} \| L f \|_{p}+\| |B|^{1/2} f \|_{p} \big{)}^{p/n}
 \end{equation}
 \begin{equation}\label{eq:czc}\qquad \int_{Q_k}| L b_{k} |^{p}+ R_{k}^{-p} | b_{k}|^{p}  \leq C\alpha^{p} | Q_{k} |
 \end{equation}
 \begin{equation}\label{eq:czd} \sum_{k}| Q_k | \leq C\alpha ^{-p}\big{(}\int_{\mathbf{R} ^{n}} | L f |^{p}+| |B|^{1/2}f |^{p} \big{)}
 \end{equation}
 \begin{equation}\label{eq:cze}\sum_{k} \mathbf{1}_{Q_k} \leq N,
 \end{equation}
where 
 $N$ depends only on  the dimension and $C$ on the dimension,  $p$ and the $RH_{n/2}$ constant of $|B|$. Here, $R_{k}$ denotes the sidelength of $Q_{k}$ and gradients are taken in the sense of distributions in $\mathbb{R}^n$.
 \end{lemm}
 \begin{rem} Note that by \eqref{eq:czb} for $p<2$, we obtain: 
 \begin{equation}\| L g \|_{2}+\| |B|^{1/2} g \|_{2} \leq C\alpha^{1-\frac{p}{2} }\big{(} \| L f \|_{p}+\| |B|^{1/2} f \|_{p} \big{)}^{p/2},
 \end{equation}
We will use this inequality to prove \ref{th:1b}.
 \end{rem}
 The rest of the section is devoted to the demonstration of Lemma \ref{th:CZ1}.
 
  \begin{proof}
Let $\Omega$ be the open set $\{ x\in \mathbb{R}^{n}; M\big{(}| L f|^{p} +| |B|^{1/2}f |^{p} \big{)}(x)>\alpha^{p} \}$, where $M$ is the uncentered maximal operator over the cubes of $ \mathbb{R}^{n}$. If $\Omega$ is empty, then set $g=f$ and $b_{i}=0$. Otherwise, our argument is subdivided into six steps.\\
 \\
 \textbf{a) Construction of the cubes:}\\
 \\
The maximal theorem gives us $$| \Omega | \leq C{\alpha}^{-p}\int_{\mathbb{R}^n} | L f|^{p} +| |B|^{1/2}f |^{p}<\infty .$$
 Let $(Q_k)$ be a Whitney decomposition of $\Omega$ by dyadic cubes so to say $\Omega$ is the disjoint union of the $Q_k$'s, 
the cubes $2Q_i$ are contained in  $\Omega$ and have the bounded overlap property, but the cubes $4Q_k$ intersect $F=\mathbb{R}^n\setminus \Omega$.\footnote{In fact, the factor 2 should be some $c=c(n)>1$ explicitely given in [\cite{St},Chapter 6]. We use this convention to avoid too many irrelevant constants.}
 
Hence $$\sum_{k} | 2 Q_k | \leq C | \Omega | \leq C{\alpha}^{-p}\int_{\mathbb{R}^n} | L f|^{p} +| |B|^{1/2}f |^{p} .$$ \eqref{eq:czd} and
 \eqref{eq:cze} are satisfied by the cubes $2Q_k$.\\
 \\
 \textbf{b) Construction of $b_{k}$:}\\ 
 \\
 Let $(\chi_k)$ be a partition of unity on $\Omega$
associated to the covering $(Q_k)$ so that for each $k$, $\chi_k$ is a $C^1$ function supported in $2Q_k$ with 
 \begin{equation}\label{part}
 \| \chi_{k} \|_\infty + R_k \| \nabla {\chi_{k}} \| _{\infty} \leq c(n),
 \end{equation}
  where $R_k$ is the sidelength of $Q_k$ and $\sum \chi_{k} =1$ on $\Omega$.
  We say that a cube $Q$ is of type 1 if 
  $R^{2}\avert {Q} |B| > 1$,
 and is of type 2 if $R^{2}\avert {Q} |B| > 1$.
  
  We apply the gauge transformation on the cubes
$ 2Q_k$ such that $ Q_k$ is of type 2, hence there exist $\textbf{h}_{k}\in L^{n}(2Q_{k},\mathbb{R}^{n})$ and a real function $\phi_{k} \in H^{1}(2Q_{k})$ such that
\begin{equation}\textbf{h}_{k}=\textbf{a}- \nabla \phi_{k} \qquad \textrm{a.e on }\,\, 2Q_{k},
 \end{equation}
 
 \begin{equation}\label{jaugecz}\big{(}\avert{2Q_{k}} | \textbf{h}_{k}|^{n}\big{)}^{1/n}\leq  c\,R_{k}\big{(}\avert{2Q_{k}} |B|^{n/2}\big{)}^{2/n}.
\end{equation}
We denote $$m_{2Q_{k}}(e^{i\phi_{k}}f)= \avert{2Q_{k}} (e^{i\phi_{k}}f).$$ 
Let
 \begin{equation}
 b_{k}=\left\{\begin{array}{ll}
 f \chi_{k}, \,& \mathrm{if}\ Q_{k}\ \mathrm{is\ of\ type\ 1}\\
 \big{(}f-e^{-i\phi_{k}}m_{2Q_{k}}(e^{i\phi_{k}}f)\big{)} \chi_{k},& \mathrm{if}\ Q_{k}\ \mathrm{is \ of \ type\  2}.

 \end{array}
 \right.
 \end{equation}
  \textbf{c) Proof of estimate \eqref{eq:czc} :}\\
  \\
  Suppose $Q_k$ is of type 1, then $$R_{k}^{-p}\leq  c\big{(}\avert {2Q_{k}} |B| \big{)}^{p/2}\leq C \avert {2Q_{k}} |B| ^{p/2},$$
  where we used $|B|^{p/2} \in RH_{2/p}$ if $p<2$ (by proposition \ref{11.1}) and the Jensen's inequality with convex function $t\longmapsto t^{p/2}$ if $p\geq 2$.
  
In order to control  $L\,b_{k}$, we have
$$L b_{k}= L( f\chi_{k})=(Lf)\chi_{k}+\frac{1}{i}f\,\nabla\chi_{k},$$
then
$$\int_{2Q_k}| L b_{k} |^{p}+ R_{k}^{-p} | b_{k}|^{p}  \leq C \| \chi_{k} \|^{p} _{\infty}\int_{2Q_k}| L f |^{p} +\| \nabla\chi_{k} \|^{p} _{\infty}\int_{2Q_k} | f |^{p} +R_{k}^{-p} \| \chi_{k} \|^{p} _{\infty}\int_{2Q_k} | f |^{p}$$$$ \leq C \{\int_{2Q_k}| L f |^{p} + R_{k}^{-p}\int_{2Q_{k}}| f |^{p} \} \leq C \{ \int_{2Q_k} | L f |^{p} + | |B|^{1/2}f |^{p}  \}\leq C\alpha^{p}  |{Q_k} |,$$
 where we used  the $L^p$ version of the  usual Fefferman-Phong inequality \eqref{eq:feff} and the intersection of $4Q_{k}$ with $F$, hence $\int_{4Q_k} | L f |^{p} + | |B|^{1/2}f |^{p}  \leq C\alpha^{p}  |{4Q_k} |$. Then estimation \eqref{eq:czc} holds for the cubes of type 1.
 
  If $Q_k$ is of type 2,  $R_{k}^{2}\avert{Q_{k}} |B| \leq 1$. $|B(x)|dx$ is a doubling measure, then there exists $C>0$, such that $R_{k}^{2}\avert{Q_{2k}} |B| \leq C$.
  $$b_{k}=\big{(}f-e^{-i\phi_{k}}m_{2Q_{k}}(e^{i\phi_{k}}f)\big{)} \chi_{k}.$$
 Let us estimate $L\,b_{k}$. By the Gauge invariance, all we require is the estimation of $\tilde{L}(e^{i\phi_{k}}b_{k})$, where $$\tilde{L}=\frac{1}{i}\nabla-\textbf{h}_{k}.$$
 We have
 $$ \tilde{L}(e^{i\phi_{k}}b_{k}) = \chi_{k}\big{(}\tilde{L} f_{k})+ \frac{1}{i}\big{(}f_{k}-m_{2Q_{k}}f_{k}\big{)} \nabla\chi _{k}-\big{(}\avert{2Q_{k}} f_{k}\big{)}\, \chi_{k}\,  \textbf{h}_{k},$$
where $f_{k}=e^{i\phi_{k}}f$. Then,
     $$\big{(}\avert{2Q_{k}} |L b_k|^{p}\big{)}^{1/p}\leq C \{ \big{(}\avert{2Q_{k}}|\tilde{L} f_{k}|^{p}\big{)}^{1/p}\, ||\chi_{k}||_{\infty}+ \big{(}\avert{2Q_{k}}|(f_{k}-m_{2Q_{k}}f_{k})|^{p}\big{)}^{1/p}\,|| \nabla\chi _{k}||_{\infty} $$$$+\big{(}\avert{2Q_{k}} |\textbf{h}_{k}|^{p}\avert{2Q_{k}}| f_{k}|^{p}\big{)}^{1/p}\, ||\chi_{k}||_{\infty}\}.$$
Using the Poincar\'e inequality and condition \eqref{part}, we obtain $$\big{(}\avert{2Q_{k}} |\tilde{L} b_k|^{p}\big{)}^{1/p}\leq C \{ \big{(}\avert {2Q_{k}}|\tilde{L} f_{k}|^{p}\big{)}^{1/p}+ \big{(}\avert {2Q_{k}}|\nabla f_{k}|^{p}\big{)}^{1/p} +\big{(}\avert{ 2Q_{k}} |\textbf{h}_{k}|^{p}\avert{2Q_{k}}| f_{k}|^{p}\big{)}^{1/p}\}$$

 $$\leq C \{ \big{(}\avert{2Q_{k}}|\tilde{L} f_{k}|^{p}\big{)}^{1/p}+ \big{(}\avert{2Q_{k}}|\frac{1}{i}\nabla f_{k}-\textbf{h}_{k}f_{k}|^{p}\big{)}^{1/p}$$$$+ \big{(}\avert{ 2Q_{k}} |\textbf{h}_{k}|^{p}\avert{2Q_{k}}| f_{k}|^{p}\big{)}^{1/p}+ \big{(}\avert{ 2Q_{k}}|\textbf{h}_{k}f_{k}|^{p}\big{)}^{1/p}\}.$$
 Hence $$\big{(}\avert{2Q_{k}} |L b_k|^{p}\big{)}^{1/p}\leq C \{ \big{(}\avert{2Q_{k}}|\tilde{L} f_{k}|^{p}\big{)}^{1/p}+ I+II\}.$$
 
 Next, we apply inequality \eqref{jaugecz} to estimate $I$. The fact that $|B|$ is  a $RH_{n/2}$ weight and $Q_{k}$ is of type $2$ leads:  
\begin{align*}
\big{(}\avert{ 2Q_{k}} |\textbf{h}_{k}|^{p}\big{)}^{1/p}\big{(}\avert{2Q_{k}}| f_{k}|^{p}\big{)}^{1/p}&\leq \big{(}\avert{ 2Q_{k}} |\textbf{h}_{k}|^{n}\big{)}^{1/n}\big{(}\avert{2Q_{k}}| f_{k}|^{p}\big{)}^{1/p}
\\
& \leq C R_{k}\big{(}\avert{ 2Q_{k}} |B|^{n/2}\big{)}^{2/n}\big{(}\avert{2Q_{k}}| f_{k}|^{p}\big{)}^{1/p}
\\
&\leq C R_{k}\big{(}\avert{ 2Q_{k}} |B|\big{)}\big{(}\avert{2Q_{k}}| f_{k}|^{p}\big{)}^{1/p}
\\
&\leq C\big{(}\avert{ 2Q_{k}} |B|\big{)}^{1/2}\big{(}\avert{2Q_{k}}| f_{k}|^{p}\big{)}^{1/p}.
\end{align*}
 By Fefferman-Phong inequality \eqref{eq:feff},
 $$I\leq C \big{(}\big{(}\avert{2Q_{k}}|B|\big{)}^{p/2}\avert{ 2Q_{k}}|f_{k}|^{p}\big{)}^{1/p}\leq C \big{(}\avert{2Q_{k}}|B|^{p/2}\, \avert{ 2Q_{k}}|f_{k}|^{p}\big{)}^{1/p} \leq C \big{(}\avert{2Q_{k}}|\tilde{L}f_{k}|^{p}+||B|^{1/2}f_{k}|^{p}\big{)}^{1/p}.$$
 Hence
 \begin{equation}\label{eqI} I\leq C \avert{2Q_{k}}|\tilde{L}f_{k}|^{p}+||B|^{1/2}f_{k}|^{p}.
 \end{equation}
 
 To estimate the second term $II$, first we use the H\"older inequality and the fact that $|B|\in RH_{n/2}$ and  $Q_{k}$ is of type $2$. Next, we apply Poincar\'e inequality and the diamagnetic inequality ( under our hypothesis, $f_{k}\in W^{1,2}_{\textbf{a}}(\mathbb{R}^{n})$):
\begin{align*}
II=\big{(}\avert{ 2Q_{k}}|\textbf{h}_{k}f_{k}|^{p}\big{)}^{1/p}&\leq \big{(}\avert{ 2Q_{k}}|\textbf{h}_{k}|^{p.n/p}\big{)}^{p/pn}\big{(}\avert{ 2Q_{k}}|f_{k}|^{p.n/(n-p)}\big{)}^{(n-p)/pn}
\\
&\leq C R_{k} \big{(}\avert{ 2Q_{k}}|B|^{n/2}\big{)}^{2/n}\big{(}\avert{ 2Q_{k}}|f_{k}|^{pn/(n-p)}\big{)}^{(n-p)/pn}
\\
&\leq C R_{k} \big{(}\avert{ 2Q_{k}}|B|\big{)}\big{(}\avert{ 2Q_{k}}|f_{k}|^{pn/(n-p)}\big{)}^{(n-p)/pn}
\\
&\leq C R_{k} \big{(}\avert{ 2Q_{k}}|B|\big{)}\{\big{(}\avert{ 2Q_{k}}||f_{k}|-m_{2Q_{k}}\big{(}|f_{k}|\big{)}|^{p.n/(n-p)}\big{)}^{(n-p)/pn}+m_{2Q_{k}}\big{(}|f_{k}|\big{)}\}
\\
&\leq C \{R_{k}^{2} \big{(}\avert{2Q_{k}}|B|\big{)}\big{(}\avert{2Q_{k}} |\tilde{L} f_{k}|^{p}\big{)}^{1/p}+\big{(}\avert{2Q_{k}}|B|\big{)}^{1/2}\big{(}\avert{ 2Q_{k}}|f_{k}|\big{)}\}
\\
&\leq C \{\big{(}\avert{2Q_{k}} |\tilde{L} f_{k}|^{p}\big{)}^{1/p}+\big{(}\avert{2Q_{k}}|\tilde{L}f_{k}|^{p}+||B|^{1/2}f_{k}|^{p}\big{)}^{1/p}\}.
\end{align*}
Then
\begin{equation}\label{eqII}II\leq C \big{(}\avert{2Q_{k}}|\tilde{L}f_{k}|^{p}+||B|^{1/2}f_{k}|^{p}\big{)}^{1/p}.
\end{equation}
 
Since $|L(f)|= |\tilde{L}(f_{k} )|$, then, by gauge invariance, 
$$\avert{2Q_{k}} |L b_k|^{p}\leq C \{ \avert{2Q_{k}}|L f|^{p}+||B|^{1/2} f|^{p}\}\leq c \alpha^{p}.$$
And by the same argument, we have
 $$R_{k}^{-p}\avert{2Q_{k}} |b_k|^{p}=R_{k}^{-p}\avert{2Q_{k}} |\big{(}f_{k}-m_{2Q_{k}}f_{k}\big{)} \chi _{k}|^{p}\leq C \alpha^{p}.$$
 
Thus \eqref{eq:czc} is proved .\\ 
\\
 \textbf{d) Definition and properties of $|B|^{\frac{1}{2}}g$:}\\ 
 \\
 Set $g=f- \sum b_{k}$. Note that, by \eqref{eq:cze}, this sum is locally finite.
 It is clear that $g=f$ on $F$ and $g=\sum_{k\in J}e^{-i\phi_{k}}m_{2Q_{k}}(e^{i\phi_{k}}f) \chi_{k}$ on $\Omega$, where $J$ is the set of indices $k$ such that $Q_{k}$ is of type 2.
$$\int_{\mathbb{R}^n} | |B|^{1/2}g|^{n}=\int_{F} | |B|^{1/2}g|^{n}+\int_{\Omega} | |B|^{1/2}g|^{n}=I+II.$$
By construction,
$$I=\int_{F} | |B|^{1/2}g|^{n}= \int_{F} | |B|^{1/2} f|^{n}\leq c \alpha^{n-p} \big{(}\| Lf \|_{p}+\| |B|^{1/2} f\|_{p}\big{)}^{p}.$$
  Since $|B|^{1/2}\in RH_{n}$, and by the $L^1$ Fefferman-Phong inequality \eqref{eq:feff} on $2Q_{k}$, type 2 cubes, we obtain  $$II=\int_{\Omega} | |B|^{1/2}. g|^{n}\leq c \sum_{k\in J}|Q_{k}|[\avert{2Q_{k}}|B|^{1/2}\avert{2Q_{k}}| f |]^{n}\leq C \sum_{k\in J}|Q_{k}|\alpha^{n}$$
  $$\qquad \leq c \alpha^{n-p} \int_{\mathbb{R}^n} |Lf |^{p}+| |B|^{1/2} f|^{p}.$$
  Hence
  \begin{equation}\label{bg}
  \big{(}\int_{\mathbb{R}^n} | |B|^{1/2}  g|^{n}\big{)}^{1/n} \leq c \alpha^{1-\frac{p}{n}} \big{(}\| Lf \|_{p}+\| |B|^{1/2} f\|_{p}\big{)}^{p/n}.
  \end{equation}
  
  \textbf{e) Estimate of $Lg$:}\\
  \\  
  Let $K$ the set of indices $k$. Let $\xi\in C^{\infty}_{0}(\mathbb{R}^{n})$, a test function. We know that, for all $k\in K$ such that $x\in 2Q_{k}$, there exists $C>0$ such that $d(x,F)> C\,R_{k}$. Therefore,
  $$\int \sum_{k \in K} |b_{k}| |\xi|\leq C \big{(} \int \sum_{k\in K} \frac{|b_{k}|}{R_{k}}\big{)} \sup_{x\in \mathbb{R}^n}\big{(} d(x,F) |\xi(x)|\big{)}.$$
The estimate \eqref{eq:czc} gives us
$$\int  |b_{k}|^p\leq C{R_k}^p\alpha^{p} |Q_{k}|.$$
Hence
$$\int \sum_{k\in K} |b_{k}| |\xi|\leq C\alpha |\Omega| \sup_{x\in \Omega}\big{(} d(x,F) |\xi(x)|\big{)}.$$
We conclude that $\sum_{k\in K} b_{k}$ converges in the sense of distributions in $\mathbb{R}^n$.\\ Then,
$$\nabla g=\nabla f-\sum_{k\in K} \nabla b_{k},\,\, \textrm{in the sense of distributions in $\mathbb{R}^n$.}$$
 Since the sum is locally finite in $\Omega$ and vanishes on $F$, then $\textbf{a}\,g=\textbf{a}\,f - \sum_{k\in K}\textbf{a}\, b_{k}$ holds always every where in $\mathbb{R}^n$. Hence
$$Lg=Lf-\sum_{k\in K} L b_{k},\,\, \textrm{a.e in $\mathbb{R}^n$.}$$   
  \\
  \textbf{f) Proof of estimate \eqref{eq:czb}:}\\
  \\  
$\sum_{k\in K} \nabla \chi_{k}(x)=0$ for all $x\in \Omega$, then
 $$Lg = (L f) \mathbf{1}_{F}+ \sum_{k\in J}L(e^{-i\phi_{k}}\,m_{2Q_k}(e^{i\phi_{k}}f) \chi_{k})\qquad \textrm{a.e in $\mathbb{R}^n$}. $$
 Since $$L( u)= e^{-i \phi_{k}} \tilde{L}(e^{i\phi_{k}} u) \quad \textrm{where}\quad \tilde{L}=\frac{1}{i} \nabla -\textbf{h}_{k},$$ then
 $$\sum_{k\in J}L(e^{-i\phi_{k}}\,m_{2Q_k}(e^{i\phi_{k}}f) \chi_{k})=\frac{1}{i}\sum_{k\in J}e^{-i\phi_{k}}m_{2Q_k}(e^{i\phi_{k}}f) \nabla\chi_{k}-\sum_{k\in J}e^{-i\phi_{k}}\,m_{2Q_k}(e^{i\phi_{k}}f) \chi_{k}\textbf{h}_{k}$$$$=G_{1}+G_{2}.\,\,\,$$
Let us estimate $\|G_2\|_{n}$. First, we use \eqref{eq:cze}:
  $$\|G_{2}\|_{n}=\big{(}\int_{\Omega}|\sum_{k\in J}m_{2Q_k}(e^{i\phi_{k}}f) \chi_{k}\textbf{h}_{k}|^{n}\big{)}^{1/n}\leq C N^{\frac{n-1}{n}}\big{(} \sum_{k\in J}\int_{2Q_{k}}|m_{2Q_k}(e^{i\phi_{k}}f\big{)}\, \textbf{h}_{k}|^{n}\big{)}^{1/n}$$
 $$\leq C N^{\frac{n-1}{n}}\big{(}\sum_{k\in J}|2Q_{k}|\avert{2Q_{k}}|\textbf{h}_{k}|^{n}\,|m_{2Q_k}(e^{i\phi_{k}}f)|^n\big{)}^{1/n} .$$
 Lemma \ref{2jaugebes} and the fact that $|B|$ is a $RH_{n/2}$ weight function and $Q_{k}$ is a type 2 cube, yield  
\begin{align*}\|G_{2}\|_{n}
&\leq C N^{\frac{n-1}{n}}\big{(}\sum_{k\in J}|2Q_{k}|R^{n}_{k}\big{(}\avert{2Q_{k}}|B|^{n/2}\big{)}^{2}\,|m_{2Q_k}(e^{i\phi_{k}}f)|^n\big{)}^{1/n}
\\
& \leq C N^{\frac{n-1}{n}}\big{(}\sum_{k\in J}|2Q_{k}|\big{(}R_{k}\avert{2Q_{k}}|B|\,|m_{2Q_k}(e^{i\phi_{k}}f)|\big{)}^n\big{)}^{1/n}
\\
 & \leq C N^{\frac{n-1}{n}}\big{(}\sum_{k\in J}|2Q_{k}|\big{(}\big{(}\avert{2Q_{k}}|B|\big{)}^{1/2}\,|m_{2Q_k}(e^{i\phi_{k}}f)|\big{)}^n\big{)}^{1/n}
\\
 &\leq C N^{\frac{n-1}{n}}\big{(}\sum_{k\in J}|2Q_{k}|\big{(}\avert{2Q_{k}}|B|^{p/2}\,\avert{2Q_{k}}|f|^{p}\big{)}^{n/p}\big{)}^{1/n} 
\\
&\leq C N^{\frac{n-1}{n}}\alpha\big{(}\sum_{k\in J}|2Q_{k}|\big{)}^{1/n} \leq C N^{\frac{n-1}{n}}\alpha^{1-\frac{p}{n}}\big{(}\int_{\mathbb{R}^{n}}|Lf |^{p}+||B|^{1/2}f|^{p}\big{)}^{1/n}.
\end{align*}
We obtain
 \begin{equation}\label{eq:g2}\|G_{2}\|_{n} \leq C \alpha^{1-\frac{p}{n}}\big{(}\|Lf \|_{p}+\||B|^{1/2}f\|_{p}\big{)}^{p/n}.
 \end{equation}

  Recall that $G_{1}(x)=\sum_{k\in J}e^{-i\phi_{k}(x)}m_{2Q_{k}}(e^{i\phi_{k}}f) \nabla \chi_{k}(x)$. We will estimate $\|G_1\|_n$.
 For all $m\in K$, set $K_{m}=\{l\in K, 2Q_{l}\cap 2Q_{m}\neq\emptyset\}.$
 By construction of Whitney cubes, there exists a constant $c>0$ (we can take $c=18$) such that for all $m\in K$ $2Q_{l}\subset c \,Q_{m}$, for all $l\in K_m$. Set $\tilde{Q}_{m}=cQ_m$,
  $$G_{1}(x)=\sum_{k\in J}e^{-i\phi_{k}(x)}m_{2Q_{k}}(e^{i\phi_{k}}f) \nabla \chi_{k}(x)= \sum_{m\in K} \chi_{m}(x)\big{(}\sum_{k\in J\cap K_{m} }e^{-i\phi_{k}(x)}m_{2Q_{k}}(e^{i\phi_{k}}f) \nabla \chi_{k}(x) \big{)}.$$
 It suffices to prove
\begin{equation}\label{G1}\int_{2Q_{m}}\left|\sum_{k\in J\cap K_{m} }e^{-i\phi_{k}}m_{2Q_{k}}(e^{i\phi_{k}}f) \nabla \chi_{k}\right|^{n}\leq C \alpha^{n} |2Q_{m}|.
\end{equation}
We fix an $m$, by the gauge transformation of corollary \ref{jaugebes1}, $\tilde{\textbf{h}}_{m}=\textbf{a}-\nabla \tilde{\phi}_{m}$ satisfies \eqref{jaugecz} on $\tilde{Q}_{m}$.\\
  \paragraph{\textit{First case:}} There exists $k_{0}\in J\cap K_{m}$ such that $2Q_{k_{0}}$ is of type 1.\\
 Since $|B(x)| dx$ is a doubling measure, there exists a constant $A>0$ which depends on $|B|$, such that for all $k\in K_{m}$, $$(2R_{k})^{2}\avert{2Q_{k}} |B| >A.$$
   $|B|^{1/2}\in RH_{2}$, which means that $R^{-1}_{k} \leq C \avert{2Q_{k}} |B|^{1/2}$, for all $k\in K_{m}$. Then
     $$\int_{2Q_{m}}\left|\sum_{k\in J\cap K_{m} }e^{-i\phi_{k}}m_{2Q_{k}}(e^{i\phi_{k}}f) \nabla \chi_{k}\right|^{n}\leq C\big(\sum_{k\in J\cap K_{m} }|Q_k|R^{-n}_{k}\big{(}\avert{2Q_{k}}|f|\big{)}^{n}\big) $$$$\leq C [\sum_{k\in J\cap K_{m} }|Q_k| R^{-n}_{m}\big{(}\avert{2Q_m}|f|\big{)}^{n}]^{1/n} \leq C |Q_{m}|  \alpha,$$ here we used $|Q_{k}|\sim |Q_{m}|$, \eqref{eq:cze}, Fefferman-Phong inequality \eqref{eq:feff} and $4Q_{m}\cap F\neq \emptyset$. \\

  \paragraph{\textit{Second case:}}$\forall k \in J\cap K_{m}$, $2Q_{k}$ is of type 2.
   \begin{align*}\sum_{k\in J\cap K_{m} }e^{-i\phi_{k}}m_{2Q_{k}}(e^{i\phi_{k}}f) \nabla \chi_{k}=
&\sum_{k\in J\cap K_{m} } \big{(}e^{-i\phi_{k}}m_{2Q_{k}}(e^{i\phi_{k}}f)-e^{-i\tilde{\phi}_{m}}m_{2Q_{k}}(e^{i\tilde{\phi}_{m}}f)\big{)} \nabla \chi_{k}
\\
&+\sum_{k\in J\cap K_{m} }e^{-i\tilde{\phi}_{m}}\big{(}m_{2Q_{k}}(e^{i\tilde{\phi}_{m}}f\big{)}-m_{\tilde{Q}_{m}}(e^{i\tilde{\phi}_{m}}f)\big{)} \nabla \chi_{k}
\\
&+\sum_{k\in J\cap K_{m} }e^{-i\tilde{\phi}_{m}}m_{\tilde{Q}_{m}}(e^{i\tilde{\phi}_{m}}f) \nabla \chi_{k}
\\
&=I+II+III.
\end{align*}
Thus
$$III=\sum_{k\in K_{m}} \chi_{m}e^{-i\tilde{\phi}_{m}}m_{\tilde{Q}_{m}}(e^{i\tilde{\phi}_{m}}f)  \nabla \chi_{k}-\sum_{k\in K_{m}\setminus J} \chi_{m}e^{-i\tilde{\phi}_{m}}m_{\tilde{Q}_{m}}(e^{i\tilde{\phi}_{m}}f)  \nabla \chi_{k}.$$
We know that $\sum_{k\in K_{m}} \nabla \chi_{k}(x)=\sum_{k\in K} \nabla \chi_{k}(x)=0,$ for all $x\in 2Q_{m}$, and hence the first term in the above expression vanishes .\\ Since $2Q_{k}$, with $k\in K_{m}\setminus J$, are type 1 cubes, then we obtain using the same procedure as in the first case
$$\int_{2Q_{m}}| III|^{n} \leq C |Q_{m}| \alpha.$$

 Now we will control the $L^{\infty}$ norm of $II$,
\begin{align*}
\left| \sum_{k\in J\cap K_{m} }e^{-i\tilde{\phi_{m}}(x)}\big{(}m_{2Q_{k}}e^{i\tilde{\phi}_{m}}f\,-m_{\tilde{Q}_{m}} e^{i\tilde{\phi}_{m}}f\big{)} \nabla \chi_{k}(x) \right|
& \leq  \sum_{k\in J\cap K_{m} }|m_{2Q_{k}}e^{i\tilde{\phi_{m}}}f-m_{\tilde{Q}_{m}}e^{i\tilde{\phi}_{m}}f| ||\nabla \chi_{k}||_{\infty}
\\
 & \leq C \sum_{k\in J\cap K_{m} } | m_{2Q_{k}}(e^{i\tilde{\phi}_{m}}f)-m_{\tilde{Q}_{m}}(e^{i\tilde{\phi}_{m}}f)| R_{k}^{-1},
 \end{align*}
since \begin{equation}\label{eq:av}
| m_{2Q_{k}}(e^{i\tilde{\phi}_{m}}f)-m_{\tilde{Q}_{m}}(e^{i\tilde{\phi}_{m}}f)| \leq C \tilde{R_{m}}\alpha,
\end{equation}
then $$| \sum_{k}e^{-i\tilde{\phi_{m}}(x)}\big{(}m_{2Q_{k}}(e^{i\tilde{\phi_{m}}}f)-m_{\tilde{Q_{m}}}(e^{i\tilde{\phi_{m}}}f)) \nabla \chi_{k}(x) | \leq C N\alpha,$$

 It suffices to prove \eqref{eq:av}:  
\begin{align*}| m_{2Q_{k}}(e^{i\tilde{\phi}_{m}}f)-m_{\tilde{Q}_{m}}(e^{i\tilde{\phi}_{m}}f)|& \leq C | m_{\tilde{Q}_{m}}(e^{i\tilde{\phi}_{m}}f-m_{2Q_{k}}(e^{i\tilde{\phi}_{m}}f))| 
\\
&\leq C R_{k} \big{(}m_{\tilde{Q}_{m}}(|\nabla (e^{i\tilde{\phi}_{m}}f)|^{p}\big{)}^{1/p} 
\\
&\leq C \tilde{R}_{m} \{\big{(}m_{\tilde{Q}_{m}}(|\tilde{L}( e^{i\tilde{\phi}_{m}}f)|)^{p}\big{)}^{1/p}+\big{(}m_{\tilde{Q}_{m}}(|\textbf{h}_{m}e^{i\tilde{\phi}_{m}}f|)^{p}\big{)}^{1/p}\} 
\\
&\leq C \tilde{R}_{m} \{\big{(}m_{\tilde{Q}_{m}}(|Lf|)^{p}\big{)}^{1/p}+\big{(}m_{\tilde{Q}_{m}}(|B^{1/2}f|^{p}\big{)}^{1/p}\} 
\end{align*}
where $\tilde{L}=\frac{1}{i} \nabla -\tilde{\textbf{h}}_{m}$ and $L( f)= e^{-i\tilde{\phi}_{m}} \tilde{L}(e^{i\tilde{\phi}_{m}}f)$.

Lastly we estimate $I$: \begin{align*}e^{-i\phi_{k}(x)}m_{2Q_{k}}(e^{i\phi_{k}}f)-e^{-i\tilde{\phi}_{m}(x)}m_{2Q_{k}}(e^{i\tilde{\phi}_{m}}f)=& e^{-i\phi_{k}(x)}\avert{2Q_{k}}e^{i\phi_{k}(y)}f(y)\,dy-e^{-i\tilde{\phi}_{m}(x)}\avert{2Q_{k}}e^{i\tilde{\phi}_{m}(y)}f(y)\,dy
\\
&=\avert{2Q_{k}}\big{(}e^{i(\phi_{k}(y)-\phi_{k}(x))}-e^{i(\tilde{\phi}_{m}(y)-\tilde{\phi}_{m}(x))}\big{)}f(y)\,dy.
\end{align*}
Next, we use the following inequality
 $$|e^{i(\phi_{k}(y)-\phi_{k}(x))}-e^{i(\tilde{\phi}_{m}(y)-\tilde{\phi}_{m}(x))}|\leq  |(\phi_{k}(y)-\phi_{k}(x))-(\tilde{\phi}_{m}(y)-\tilde{\phi}_{m}(x))|,$$
and we obtain
 $$|e^{i(\phi_{k}(y)-\phi_{k}(x))}-e^{i(\tilde{\phi}_{m}(y)-\tilde{\phi}_{m}(x))}|\leq |\big{(}\phi_{k}-\tilde{\phi}_{m})(y)-m_{2Q_{k}}(\phi_{k}-\tilde{\phi}_{m})+m_{2Q_{k}}(\phi_{k}-\tilde{\phi}_{m})-\big{(}\phi_{k}-\tilde{\phi}_{m})(x)|.$$
 Therefore
 $$\int_{2Q_{k}}\big{|}\avert{2Q_{k}}|e^{i(\phi_{k}(y)-\phi_{k}(x))}-e^{i(\tilde{\phi}_{m}(y)-\tilde{\phi}_{m}(x))}\,f(y)|dy\big{|}^{n}dx$$$$\leq |2Q_k|\big{[}\avert{2Q_{k}}|f(y)||\big{(}\phi_{k}-\tilde{\phi}_{m})(y)-m_{2Q_{k}}\big{(}\phi_{k}-\tilde{\phi}_{m})|dy\big{]}^{n}$$$$+ \{\avert{2Q_{k}}|f(y)|dy\}^{n}.\int_{2Q_{k}} |\big{(}\phi_{k}-\tilde{\phi}_{m})(x)-m_{2Q_{k}}(\phi_{k}-\tilde{\phi}_{m})|^{n}dx= |2Q_k|X^{n}+Y.$$
We apply the H\"older and Poincar\'e inequalities. Then, we use \eqref{jaugecz}, and the fact that $|B|$ is in $RH_{n/2}$ and $2Q_k$ is a of type 2.

 $$X\leq\big{(}\avert{2Q_{k}}|f(y)|^{\frac{n}{n-1}}\, dy\big{)}^{\frac{n-1}{n}}\, \big{(}\avert{2Q_{k}}|\big{(}\phi_{k}-\tilde{\phi}_{m})(y)-m_{2Q_{k}}\big{(}\phi_{k}-\tilde{\phi}_{m})|^{n}dy\big{)}^{\frac{1}{n}}$$
 $$\leq C R_{k} \big{(}\avert{2Q_{k}}|f(y)|^{\frac{n}{n-1}}\, dy\big{)}^{\frac{n-1}{n}}\, \big{(}\avert{2Q_{k}}|\nabla(\phi_{k}-\tilde{\phi}_{m})(y)|^{n}dy\big{)}^{\frac{1}{n}}.$$
Moreover, by construction
 $$\nabla(\phi_{k}-\tilde{\phi_{m}})=\tilde{\textbf{h}}_{m}-\textbf{h}_{k},$$
then
 \begin{align*}
X\leq & C R_{k} \big{(}\avert{2Q_{k}}|\big{(}\tilde{\textbf{h}}_{m}-\textbf{h}_{k})(y)|^{n}dy\big{)}^{\frac{1}{n}}\big{(}\avert{2Q_{k}}|f(y)|^{\frac{n}{n-1}}\, dy\big{)}^{\frac{n-1}{n}} 
 \\
&\leq C R_{k}  \big{(}\avert{2Q_{k}}|\big{(}\tilde{\textbf{h}}_{m}-\textbf{h}_{k})(y)|^{n}dy\big{)}^{\frac{1}{n}}\big{(}\avert{2Q_{k}}|f(y)|^{\frac{n}{n-1}}\, dy\big{)}^{\frac{n-1}{n}} 
\\
 &\leq  C R^{2}_{k}\avert{2Q_{k}}|B| [\big{(}\avert{2Q_{k}}||f(y)|-m_{2Q_{k}}(|f|)|^{\frac{n}{n-1}}\, dy\big{)}^{\frac{n-1}{n}}+ C m_{2Q_{k}}(|f|)] 
\\
 &\leq   C R^{2}_{k}\avert{2Q_{k}}|B| [\avert{2Q_{k}}|Lf(y)|\, dy+m_{2Q_{k}}(|f|)] \leq   C [\alpha |Q_{k}|^{1/n}+R^{2}_{k}\avert{2Q_{k}}|B|\avert{2Q_{k}}|f|]
\\
 &\leq  C R_{k} [\alpha +\big{(}\avert{2Q_{k}}|Lf(y)|+||B|^{1/2}f(y)|\, dy\big{)}]\leq C R_{k} \alpha.
\end{align*}
We use the same arguments to estimate $Y$: 
 \begin{align*}Y=& \{\avert{2Q_{k}}|f(y)|dy\}^{n}\,\int_{2Q_{k}} |\big{(}\phi_{k}-\tilde{\phi}_{m})(x)-m_{2Q_{k}}(\phi_{k}-\tilde{\phi}_{m})|^{n}dx
\\
&\leq C R^{n}_{k}  \{\avert{2Q_{k}}|f(y)|dy\}^{n}\,\int_{2Q_{k}} |\nabla(\phi_{k}-\tilde{\phi}_{m})|^{n}
\\
&\leq C R^{n}_{k}|Q_{k}| \avert{2Q_{k}} |\tilde{\textbf{h}}_{m}-\textbf{h}_{k}|^{n}\{\avert{2Q_{k}}|f(y)|dy\}^{n}
\\
&\leq |Q_{k}|R^{n}_{k}\{ R_{k}\avert{2Q_{k}} |B|\avert{2Q_{k}}|f(y)|dy\}^{n}
\\   
& \leq |Q_{k}|R^{n}_{k}\{\avert{2Q_{k}}|Lf(y)|+||B|^{1/2}f(y)|\, dy\}^{n} \leq |Q_{k}|R^{n}_{k}\alpha^{n}.
\end{align*}
We obtain
 \begin{align*}\int_{Q_{m}}|I|^{n}\leq & C \sum_{k\in J\cap K_{m} } \int_{2Q_{k}} |\big{(}e^{-i\phi_{k}(x)}m_{2Q_{k}}(e^{i\phi_{k}}f)-e^{-i\tilde{\phi}_{m}(x)}m_{2Q_{k}}(e^{i\tilde{\phi}_{m}}f)\big{)}\nabla \chi_{k}(x)|^{n}dx
\\
&\leq C \sum_{k\in J\cap K_{m} }R^{-n}_{k}|Q_{k}|\, R^{n}_{k}\alpha^{n}\leq C \alpha \sum_{k\in J\cap K_{m} }|Q_{k}|\leq C |Q_{m}| \alpha .
\end{align*}
By integration on $\Omega$ and using \eqref{eq:czd}, we get
\begin{equation}\label{eq:g1}\|G_{1}\|_{n} \leq C \alpha^{1-\frac{p}{n}}\big{(}\|Lf\|_{p}+ \| |B|^{1/2}f\|_{p}\big{)}^{p/n}.
\end{equation}
 $L g = (L f) \mathbf{1}_{F}+ G_{1}+G_{2},\, \,a.e$ .
 
Since $| L f| \leq C\alpha$ on $F$, then estimates \eqref{eq:g1} and \eqref{eq:g2} imply
  \begin{equation}\|Lg\|_{n} \leq C \alpha^{1-\frac{p}{n}}\big{(}\|Lf \|_{p}+\||B|^{1/2}f\|_{p}\big{)}^{p/n}.
 \end{equation}
 Then  $$\| L g\|_{n}+\| |B|^{1/2} g\|_{n} \leq C \alpha^{1-\frac{p}{n}}\big{(}\|Lf \|_{p}+\||B|^{1/2}f\|_{p}\big{)}^{p/n}.$$
 Thus \eqref{eq:czb} is proved.
\end{proof}

  \subsection{Estimates for weak solution}
  Throughout this section we will assume that \textbf{$u$ is a weak solution of $H(\textbf{a},0)u=0$ in $4Q$}, where $Q$ is a cube centred at $x_{0}\in \mathbb{R}^{n}$ with sidelength $R$. The constants are independant of $u$ and $Q$.

\begin{lemm}\label{lemm:sol}(Lemma 1.11\cite{Sh4})
Let $B$ satisfying \eqref{eq:shen}. Then, for all $k>0$, there exists a constant $C_{k}>0$ such that 
 \begin{equation}|u(x_{0})| \leq \frac{C_{k}}{\{1+
Rm(x_{0},|B|)\}^{k} }\big(\avert{Q(x_{0},R)} |u|^{2}\big)^{1/2}. 
\end{equation} 
 
  \end{lemm}
	
This lemma leads to the following proposition:
  \begin{pro}\label{pro:mu} Under the hypothesis \eqref{eq:shen}, for all $q>2$, there exists a constant $C>0$ such that 
 \begin{equation}
 \big{(} \avert{Q}| m(.,|B|) u |^{q} \big{)}^{1/q}\leq C \big{(} \avert{3Q}| m(.,|B|) u |^{2} \big{)}^{1/2}.
 \end{equation}

 \end{pro}
 \begin{proof}
 Fix $q>2$
 $$\big{(} \avert{Q}| m(x,|B|) u(x) |^{q}dx \big{)}^{1/q}\leq  \{1+ R m(x_{0},|B|)\}^{k_{0} }m(x_{0},|B|)\big{(} \avert{Q}|  u |^{q} \big{)}^{1/q}$$$$\leq \frac{C_{k} \{1+ R m(x_{0},|B|)\}^{k_{0}}m(x_{0},|B|)}{\{1+ R m(x_{0},|B|)\}^{k}}\big{(} \avert{3Q}|u |^{2} \big{)}^{1/2}$$$$\leq \{1+ R m(x_{0},|B|)\}^{k_{0}-k+(k_{0}/k_{0}+1)}\frac{C_{k} m(x_{0},|B|)}{\{1+ R m(x_{0},|B|)\}^{k_{0}/k_{0}+1}}\big{(} \avert{3Q}|  u |^{2} \big{)}^{1/2}$$
 $$\leq  C\big{(} \avert{3Q}|m(.,|B|)  u |^{2} \big{)}^{1/2}.$$
Here we used Lemma \ref{th:mprop} and the fact that $ u $ satisfies Lemma \ref{11.1} with arbitrary $ k $.
\end{proof}
 
 \begin{lemm}\label{lemm 2.3}(Lemma 2.7 \cite{Sh4})
 Suppose $B$ satisfies \eqref{eq:shen}.
 For any integer $k>0$, there exists $C_k>0$, such that
 \begin{equation}\label{eq:1}| Lu(x_0)|\leq \frac{C_{k}}{\{1+Rm(x_0,|B|)\}^k}\, \frac{1}{R}\, \big{(} \frac{1}{|Q(x_{0},2R)|}\int_{Q(x_{0},2R)}|u|^{2}  \big{)}^{1/2}
 \end{equation}
 \end{lemm}
 
 \begin{rem}
	
The proof of this lemma is based on the following inequality interesting in its own right:
 
 If $2\leq p <q \leq \infty $ and $1/q- 1/p>-2/n$, then 
 \begin{equation}\label{eq:2}
 \big{(}  \avert{\frac{1}{32}\,Q}| L u |^{q} \big{)}^{1/q}\leq C \big{(}  \avert{\frac{1}{4}\,Q}| L u |^{2} \big{)}^{1/2}+C R^{2}\big{(}  \avert{\frac{1}{4}\,Q}( |\nabla B || u |)^{p} \big{)}^{1/p}
  \end{equation}
  $$+C R^{2}\big{(}  \avert{\frac{1}{4}\,Q}( | B || Lu |)^{p} \big{)}^{1/p}.$$
  \end{rem}
  \begin{rem}(\cite{Sh4}) Let $\Gamma_{0}(x,y)$ be the kernel of $H(\textbf{a},0)^{-1}$. Under assumptions \eqref{eq:shen}, for all $k>0$, there exists a constant $C_{k}>0$ such that
 \begin{equation}\label{FS} |L^{x}_{j} \Gamma_{0}(x,y)|\leq \frac{C_k}{\{1+|x-y|m(x,|B|)\}^{k}}\, \frac{1}{|x-y|^{n-1}},
 \end{equation}
 for all $x,\, y\in \mathbb{R}^{n},\, x\neq y$, where $L^{x}_{j}=\frac{1}{i}\frac{\partial}{\partial  x_{j}}-a_{j}(x)$.
 \end{rem}
  	
Using inequalities \eqref{eq:1} and \eqref{eq:2}, we obtain the following technical lemma, necessary for the proof of Theorem \ref{th:1a}:
 \begin{lemm}\label{th:RH1} Under assumptions \eqref{eq:shen}, for any $q>2$, there exists  a constante $C=C_{q}>0$ such that 
 \begin{equation}
  \big{(} \avert{Q}| L u |^{q}\big{)}^{1/q}\leq C \big{(} \avert{3Q}| L u |^{2}+| m(.,|B|) u |^{2} \big{)}^{1/2},
  \end{equation}
  and
  \begin{equation}
  | L u(x_{0}) |\leq C \big{(} \avert{3Q}| L u |^{2}+| m(.,|B|) u |^{2}\big{)}^{1/2}.
  \end{equation}
  \end{lemm}
 \begin{proof}According to the type of the cube $Q$, we would use \eqref{eq:1} or  \eqref{eq:2} to prove our lemma.

  \textbf{First case: $R^{2}\avert{Q}|B|\leq 1$}.
  
  By the definition of $m(.,|B|)$, it follows that $R \leq \frac{1}{m(x_{0},|B|)}$.
   Using \eqref{eq:shen} and
 \eqref{eq:2} we have for all $2\leq p <q \leq \infty $ and $1/q- 1/p>-2/n$
 $$\big{(}  \avert{\frac{1}{32}Q}| L u |^{q} \big{)}^{1/q}\leq C \big{(}  \avert{\frac{1}{4}Q}| L u |^{2} \big{)}^{1/2}+C R^{2}\big{(}  \avert{\frac{1}{4}Q}(   |m(x,|B|)^{3}\, u(x) |)^{p}\,dx \big{)}^{1/p}$$
  $$+C R^{2}\big{(}  \avert{\frac{1}{4}Q}( |m(x,|B|)^{2}\, Lu(x) |)^{p}\,dx \big{)}^{1/p}.$$
  Since $R < \frac{1}{m(x_{0},|B|)}$, then by the Lemma \ref{th:mprop}, $$\forall x\in Q, \, m(x,|B|)\approx m(x_{0},|B|).$$
   Hence:
   $$ \big{(}  \avert{\frac{1}{32}\,Q}| L u |^{q} \big{)}^{1/q}\leq C \big{(}  \avert{\frac{1}{4}Q}| L u |^{2} \big{)}^{1/2}+C R^{2}\,m(x_{0},|B|)^{2}\big{(}  \avert{\frac{1}{4}\,Q}(   |m(x,|B|) u(x) |)^{p}\,dx \big{)}^{1/p}$$
  $$+C R^{2}m(x_{0},|B|)^{2}\big{(} \avert{\frac{1}{4}Q} | Lu |^{p} \big{)}^{1/p}.$$
  We control $R$ by $ \frac{1}{m(x_{0},|B|)}$ and we obtain 
  $$ \big{(}  \avert{\frac{1}{32}\,Q}| L u |^{q} \big{)}^{1/q}\leq C\{  \big{(}  \avert{\frac{1}{4}\,Q}| L u |^{2} \big{)}^{1/2}+\big{(} \avert{\frac{1}{4}\,Q}(   |m(.,|B|) u |)^{p} \big{)}^{1/p}+\big{(} \avert{\frac{1}{4}\,Q} | Lu |^{p} \big{)}^{1/p}\}.$$
By iterating the inequality \ref{pro:mu}, it follows that for any $2<q\leq +\infty$,
  $$ \big{(}  \avert{\frac{1}{32}\,Q}| L u |^{q} \big{)}^{1/q}\leq C\{  \big{(}  \avert{\frac{1}{2}\,Q}| L u |^{2} \big{)}^{1/2}+\big{(} \avert{\frac{1}{2}\,Q}(   |m(.,|B|)\, u |)^{2} \big{)}^{1/2}\}.$$

 \textbf{Second case: $R^{2}\avert{Q}|B|> 1$.}
 We use Lemma \ref{lemm 2.3} to get the following inequality:
 $$| Lu(x_0)|\leq  \frac{C}{R}\, \big{(} \avert{2Q}|u|^{2}  \big{)}^{1/2}.$$
 Now we apply Fefferman-Phong inequality \eqref{eq:feff}.
 As, $$\min(\avert{2Q}|B|, \frac{1}{R^{2}})\sim \min(\avert{Q}|B|, \frac{1}{R^{2}})= \frac{1}{R^{2}}.$$
 The inequality takes the following form
  $$| Lu(x_0)|\leq C  \big{(} \avert{Q(x_{0},2R)}|Lu|^{2}+ |B||u|^{2} \big{)}^{1/2}\leq C\,\big{(} \avert{2Q}|Lu|^{2}+ |m(.,|B|)u|^{2} \big{)}^{1/2}. $$
 The last step uses \eqref{eq:shen}.
  \end{proof}
 \subsubsection{Some important tools}

Reverse H\"older inequalities previously established will be used to prove the Theorem \ref{th:1a}. 
The primary tool is the following  criterion for $L^p$ boundedness (\cite{AM1}). A slightly weaker version appears in Shen \cite{Sh2}.
 \begin{theo} \label{2theor:shen}
 
 Let $1\le p_{0} <q_0\le \infty$. Suppose that $T$ is a bounded
sublinear operator  on $L^{p_{0}}(\mathbb{R}^n)$. Assume that there exist
constants $\alpha_{2}>\alpha_{1}>1$, $C>0$ such that
\begin{equation}\label{2T:shen}
\big(\avert {Q} |Tf|^{q_0}\big)^{\frac1{q_0}}
\le
C\, \bigg\{ \big(\avert {\alpha_{1}\, Q}
|Tf|^{p_0}\big)^{\frac1{p_0}} +
(S|f|)(x)\bigg\},
\end{equation}
for all cube $Q$, $x\in Q$ and  all $f\in L^{\infty}_{comp}(\mathbb{R}^n)$ with 
support in $\mathbb{R}^n\setminus \alpha_{2}\, Q$, where $S$ is a positive operator.
Let $p_{0}<p<q_{0}$. If $S$ is bounded on $L^p(\mathbb{R}^n)$, then, there is a constant $C$ such that
$$
\|T f\|_{p}
\le
C\, \|f\|_{p}
$$
for all    $f\in L_{comp}^{\infty}(\mathbb{R}^n)$.\end{theo}

 An important step to prove the $L^p$ boundedness of Riesz transforms via the application of the previous theorem, is the control of the term $ m (. | B |)u $ on the reverse H\"older type estimates established earlier. The following result enables such a control:
 \begin{theo}\label{th:mH} Under assumptions \eqref{eq:shen}, for all $1<p<\infty$, there exists a constant $C>0$, depending on $B$, such that 
 \begin{equation}\| m(.,|B|) H(\textbf{a},0)^{-1/2}(f)\|_{p}\leq C \|
f\|_{p}, 
\end{equation}  for all $f\in
C_{0}^{\infty}(\mathbb{R}^{n}).$
 \end{theo} 
  This result is a consequence of the $L^p$ boundedness of $m(.,|B|)^{2}H(\textbf{a},0)^{-1}$ for all $1<p<\infty$ ( see Theorem 3.1\cite{Sh4}). 
  We shall use complex interpolation  relying on the fact that for all $y\in \mathbb{R}$, the imaginary power of Schr\"odinger operator  $H^{iy}$  has a bounded extension on $\mathbb{R}^n$, $1<p<\infty$. This result due to Hebisch \cite{H} follows from the Gaussian estimates on the heat kernel $e^{-tH}$ proved by \cite{DR} .
Here, $H^{iy}$ is defined  as  a bounded operator on $L^2(\mathbb{R}^n)$ by functional calculus ( see \cite{AB} for more details).

 \begin{rem}\label{rq}Under assumptions \eqref{eq:Shen1}, it is clear that $V H(\textbf{a}, V )^{-1}$ and $H(\textbf{a}, 0)H(\textbf{a}, V )^{-1}$ are $L^p$ bounded for all $1\leq p<\infty$.
\end{rem}

\subsection{Proof of Theorem \ref{th:1a} }

It is known that $LH(\textbf{a},0)^{-1/2}$ is $L^p$ bounded for all $p\leq 2$. Thus, we consider $p>2$. We need the following lemma before we start the proof of our theorem:
 \begin{lemm}\label{red} Under assumption \eqref{eq:shen}, the $L^p$ boundedness of $L H(\textbf{a},0)^{-1/2}$ is equivalent to that of $L H(\textbf{a},0)^{-1} L^{\star}$ and $L H(\textbf{a},0)^{-1}\,m(.,|B|)$.
 \end{lemm}
 \begin{proof}
If $LH(\textbf{a},0)^{-1/2}$ is $L^p$ bounded. By \cite{Sik} and \cite{DOY}, $LH(\textbf{a},0)^{-1/2}$ is $L^p$ bounded for all $1<p\leq 2$. By duality, $H(\textbf{a},0)^{-1/2}L^{\star}$ is then $L^{q}$ bounded for all  $q\geq 2$. Hence, $L H(\textbf{a},0)^{-1} L^{\star}$ is $L^p$ bounded.
 Due to the Theorem \ref{th:mH}, $H(\textbf{a},0)^{-1/2}m(.,|B|)$ is $L^p$ bounded, then $LH(\textbf{a},0)^{-1}m(.,|B|)$ is also $L^p$ bounded.
 
Reciprocally, if $LH(\textbf{a},0)^{-1}L^{\star}$ and $LH(\textbf{a},0)^{-1}m(.,|B|)$ are $L^p$ bounded, then their adjoints $LH(\textbf{a},0)^{-1}L^{\star}$ and $m(.,|B|)H(\textbf{a},0)^{-1}L^{\star}$ are bounded on $L^{p'}$.

Thus, if $\textbf{F}\in C^{\infty}_{0}(\mathbb{R}^n, \mathbb{C}^{n})$,  $\|H(\textbf{a},0)^{-1/2}L^{\star}\textbf{F}\|_{p'}=\|H(\textbf{a},0)^{1/2}H(\textbf{a},0)^{-1}L^{\star}\textbf{F}\|_{p'}$, where we used assumption \eqref{eq:shen} and inequality \eqref{eq:rev Lp}, and thus we obtain
$$\|H(\textbf{a},0)^{-1/2}L^{\star}\textbf{F}\|_{p'}\leq C\|LH(\textbf{a},0)^{-1}L^{\star}\textbf{F}\|_{p'}+\|m(.,|B|)H(\textbf{a},0)^{-1}L^{\star}\textbf{F}\|_{p'}\leq C \|\textbf{F}\|_{p'}.$$
Hence, $LH(\textbf{a},0)^{-1/2}$ is $L^p$ bounded.

 \end{proof}
 We will need the following result: 
 \begin{pro} \label{2thC'}Under assumption \eqref{eq:shen} for all $2<p<\infty$ there exists $C_{p}$ such that for any $f\in C^\infty_{0}(\mathbb{R}^n,\mathbb{C})$ and any $\textbf{F}\in C_{0}^\infty(\mathbb{R}^n, \mathbb{C}^n)$,
 $$
 \|m(.,|B|) \,H(\textbf{a},0)^{-1}m(.,|B|)\,f\|_p  \le C_{p} \|f\|_{p}, \,\, \textrm{and}\,\, \|m(.,|B|)\, H(\textbf{a},0)^{-1} L^{\star} \textbf{F}\|_p \le C'_p\|\textbf{F}\|_p.
 $$
 \end{pro}
 \begin{proof}This is a direct consequence of Theorem \ref{th:mH} and the $L^p$ boundedness of $LH(\textbf{a},0)^{-1/2}$ for all $1<p\leq 2$ .
 \end{proof}
 
  It suffices therefore to prove the following result:
 \begin{pro} \label{thD'}Under assumption \eqref{eq:shen}, for all $2<p<\infty$, there exists $C_{p}$ such that for any $f\in C^\infty_{0}(\mathbb{R}^n,\mathbb{C})$ and any $\textbf{F}\in C_{0}^\infty(\mathbb{R}^n, \mathbb{C}^n)$,
 $$
 \|L H(\textbf{a},0)^{-1}m(.,|B|)\,f\|_p  \le C_{p} \|f\|_{p}, \quad \textrm{and}\quad \|L H(\textbf{a},0)^{-1} L^{\star} \textbf{F}\|_p \le C_p\|\textbf{F}\|_p.
 $$
 \end{pro}
 \begin{proof}  Fix a cube $Q$ and let $\textbf{F}\in C^{\infty}_{0}(\mathbb{R}^n,\mathbb{C}^{n})$ supported away from $4Q$. Set $H=H(\textbf{a},0)$.  $u=H^{-1} L^{\star}\textbf{F}$ is well defined on $\mathbb{R}^n$. In particular, the support condition on $\textbf{F}$ implies that $u$ is a weak solution of $Hu=0$ in $4Q$. Hence $|u|^2$ is subharmonic on $4Q$, and by Lemma \ref{th:RH1}, we obtain that for all $q>2$, there exists a constant $C>0$ such that
 \begin{equation}
  \big{(} \avert{Q}| LH^{-1}L^{\star}\textbf{F} |^{q} \big{)}^{1/q}\leq C \big{(} \avert{3Q}| L H^{-1} L^{\star}\textbf{F}|^{2}+| m(.,|B|)H^{-1} L^{\star}\textbf{F} |^{2} \big{)}^{1/2}.
  \end{equation}
  Thus \eqref{2T:shen} holds with $T=LH^{-1}L^{\star},\, q_{0}=q,\, p_{0}=2$ and $$S\textbf{F}=\big(M(|m(.,|B|)H^{-1}L^{\star})\textbf{F}|)^{2}\big)^{\frac{1}{2}},$$ where $M$ is the maximal Hardy-Littlewood operator. Since $S$ is $L^p$ bounded for all $2<p<\infty$, then  by proposition \ref{2thC'}, $T$ is $L^p$ bounded by Theorem \ref{2theor:shen}.

We use the same argument for $L H^{-1} m(.,|B|)$. 
\end{proof}
 \textit{Proof of Theorem \ref{Maxshen} with $V=0$:}\\ Set $H_{0}=H(\textbf{a},0)$ and $m=m(. ,|B|)$. 

 $$L_{s}L_{k} H^{-1}_{0}= L_{s} H^{-1}_{0} L_{k} + L_{s} [L_{k},H^{-1}_{0}].$$
  Let $j\geq 1$, $L_{j}H^{-1/2}_{0}$ is $L^p$ bounded for all $1<p<\infty$, then $L_{s} H^{-1}_{0} L_{k}$ is $L^p$ bounded for $1<p< \infty$. We know that 
 $$[L_{k},H^{-1}_{0}]= -H^{-1}_{0} [L_{k},H_{0}]H^{-1}_{0}$$
 $$[L_{k},H_{0}]=b_{kj}L_{j}- \partial_{j}b_{kj}  $$

 $$L_{s} H^{-1}_{0} b_{kj} L_{j} H^{-1}_{0}=L_{s} H^{-1}_{0}m  \frac{b_{kj}}{m^{2}} m L_{j} H^{-1}_{0}$$
 $$L_{s} H^{-1}_{0} \partial_{j}b_{kj}  H^{-1}_{0}=L_{s} H^{-1}_{0}m\frac{\partial_{j}b_{kj}}{m^{3}}m^{2}  H^{-1}_{0}.$$
 Here, $b_{kj}$ and $\partial_{j} b_{kj}$ are the operators of multiplication by $ b_{kj}$ et $\partial_{j}b_{kj}$.
 
 Next, we use the assumptions $|b_{kj}|\leq Cm^{2}$ and $|\partial_{j}b_{kj}|\leq C m^{3}$ and the fact that $L_{s}H^{-1}_{0}m$, $mL_{j}H^{-1}_{0}$ and $m^{2}H^{-1}_{0}$ are $L^p$ bounded for all $p>1$. Thus, $L_{s} H^{-1}_{0} b_{kj} L_{j} H^{-1}_{0}$ and $L_{s} H^{-1}_{0} \partial_{j}b_{kj} L_{j} H^{-1}_{0}$ are $L^p$ bounded. Hence, $L_{s} [L_{k},H^{-1}_{0}]$ is $L^{p}$ bounded.
The $L^{p}$ boundedness of $L_{s}L_{k}H_{0}^{-1}$, for all $1<p<\infty$, follows easily..

 \section{Schr\"odinger operator with electic potential on $A_{\infty}$}
 	
In this section, we will add the electric potential $ V $ to the pure magnetic Schr\"odinger operator previously studied.
If we take some sharp hypothesis on $V$, as condition \eqref{eq:Shen1}, the approach to study the Riesz transforms will be identical, all we have to do is to replace  the weight function $|B|$ by $V+|B|$ and then Theorem \ref{Maxshen} easily follows.
 Now a natural step is to improve the conditions on $V$ and extend this result to the Scr\"odinger operators with an electric potential contained in $A_{\infty}$.
  	
To prove such a result, we will start by giving some reverse H\"older type estimates of weak solutions. 	
We will also use the reverse inequalities of Theorem \ref{th:2b}, which are always established through Calder\`on-Zygmund decomposition  similar to section 3.1.
 We  will use an equivalent approach to that of \cite{AB}. We study $H(\textbf{a},V)$ considering it as a "perturbation" of $H(\textbf{a},0)$.
 By the Kato-Simon inequality, we will establish some maximal estimates using the $L^p$ boundedness of operators $V(-\Delta +V)^{-1}$ and $\Delta(-\Delta +V)^{-1}$ proved in \cite{AB}.

 \subsection{Estimates for weak solution}\label{sec:weaksol}
 Fix an open set $\Omega$. A subharmonic function on $\Omega$ is a function $v\in L^{1}_{loc}(\Omega)$ such that $\Delta v \ge 0$ in $D'(\Omega)$.

 \begin{lemm}\label{2subh}Suppose $\textbf{a}\in  L^{2}_{loc}(\mathbb{R}^{n})^{n}$ and $0 \leq V
\in  L^{1}_{loc}(\mathbb{R}^{n})$. If $u$ is a weak solution of
$H(\textbf{a},V)u=0$ in $\Omega$, then $|u|^{2}$ is a subharmonic function and 
\begin{equation}\Delta  |u|^{2}= 2 |Lu|^{2} + 2 V |u|^{2}.
\end{equation}
\end{lemm} 
\begin{proof}Since
$$\Delta |u|^{2} =\Delta (u \overline{u})=2 Re((\Delta u) \overline{u}) + 2 |\nabla u|^{2}, $$
and
$H(\textbf{a},V)u=0$, then
$$\Delta u = \sum^{n}_{k=1}( i a_{k}\frac{\partial u}{\partial x_{k}}  + i\frac{\partial}{\partial x_{k}}(a_{k}u))+ |\textbf{a}|^{2}u+ Vu.$$
It follows that
\begin{align*}\Delta |u|^{2}&=2 Re\bigg{(}\sum^{n}_{k=1}( i a_{k}\frac{\partial u}{\partial x_{k}}+ i\frac{\partial}{\partial x_{k}}(a_{k}u))\, \overline{u}+ |\textbf{a}|^{2}u\overline{u}+ Vu \overline{u}\bigg{)}+ 2 |\nabla u|^{2} 
\\
&=2 Re\bigg{(}\sum^{n}_{k=1}( i a_{k}\frac{\partial u}{\partial x_{k}} \,\overline{u}+i \frac{\partial}{\partial x_{k}}(a_{k}u)\,\overline{u}\bigg{)}+ 2|\textbf{a}|^{2}|u|^{2}+ 2V|u|^{2}+ 2 |\nabla u|^{2} 
\\
&=2 Re\bigg{(}\sum^{n}_{k=1}( i a_{k}\frac{\partial u}{\partial x_{k}}\,\overline{u}+ i\frac{\partial}{\partial x_{k}}(a_{k}|u|^{2})-i a_{k} u\frac{\partial \overline{u}}{\partial x_{k}}\bigg{)}+ 2|\textbf{a}|^{2}|u|^{2}+ 2V|u|^{2}+ 2 |\nabla u|^{2}
\\
&=4 Im(\textbf{a} \nabla u \overline{u})+ 2|\textbf{a}|^{2}|u|^{2}+ 2 |\nabla u|^{2}+ 2|Vu|^{2}= 2 |Lu|^{2} + 2V|u|^{2} .
\end{align*}

\end{proof}
The main technical lemma is interesting in its own right. For a detailed proof see \cite{Buc} and \cite{AB}. It states that a form of the  mean value inequality for subharmonic functions still holds if the Lebesgue measure is replaced 
by a weighted measure of Muckenhoupt type. More precisely, 
\begin{lemm}\label{2th:sousharm}
 	Let  $\omega\in RH_{q}$ for some $1<q\le \infty$ and let  $0<s<\infty$ and $r>q$ (if $q=\infty$, $r=\infty$) such that $\omega\in RH_{r}$.  Then there exists a constant  $C\ge 0$  depending only on $\omega$,$r$,$p$,$s$ and $n$,  such that  for any  cube $Q$  and any nonnegative  subharmonic function  $f$ in  a neighborhood of $\overline{2Q}$  we have  for all $1< \mu \le 2$, 
$$
\big(\avert{Q} (\omega f^s)^{r} \big)^{1/r} \le C \, \avert{\mu Q} \omega f^s,\,\textrm{for}\, r<+\infty.
$$
And
$$
\sup_{Q}f \le \frac{ C}{\avert{Q}\omega} \, \avert{\mu Q} \omega f^s,\,\textrm{for}\, r=+\infty.
$$
\end{lemm}

\textbf{Throughout this section we will assume $
 V\in RH_{q}$ with $1<q\leq +\infty$ and $B$ satisfies the assumption \eqref{eq:shen} and $u$ is a weak solution of $H(\textbf{a},V)u=0$ in $4Q$}. All the constants are independant of $Q$ and $u$ but they may depend on $V$ and $q$. 
 
  First we give three important results that are the main tools for the proof of Theorem \ref{th:1a}:
  \begin{pro}\label{th:csv'} There exists a constant $C>0$ such that 
\begin{equation} \big{(} \avert{Q}| V^{1/2} u |^{2q} \big{)}^{1/2q}\leq C \big{(} \avert{3Q}| V^{1/2} u |^{2} \big{)}^{1/2}.
\end{equation}

  \end{pro}
  \begin{proof}It follows directly from Lemma \ref{2th:sousharm} and \ref{2subh}.
  \end{proof}
\begin{pro}
\label{lemm:5}  
Set $\tilde q = \inf (q^*, 2q)$. For all $1< \mu\le 2$ and $k>0$, there is a constant $C$ such that
$$
\big( \avert{ Q}   |L u|^{\tilde q} \big)^{1/\tilde q} \le   
\frac {C} { (1+ R^2\avert{Q}V)^k} \,
 \big(\avert{\mu Q}   |L u|^2 + |m(.,|B|) u|^{2} +V|u|^2 \big)^{1/2}. 
$$
\end{pro}

\begin{pro}
\label{lemm:6}  
Let $n/2 \le q <n$, for all $1< \mu\le 2$, there is a constant $C$ such that
$$
\big( \avert{ Q}   |L u|^{ q^*} \big)^{1/ q^*} \le  C \, \big(\avert{\mu Q}   |L u|^{2q} +|m(.,|B|)\,u|^{2q}\big)^{1/2q}. 
$$
If $q \ge n$ then there is a constant $C$ such that 
  $$
 \sup_{ Q}   |L u|  \le  C \, \big(\avert{\mu Q}   |L u|^{2q}+|m(.,|B|)\,u|^{2q}\big)^{1/2q}. 
$$
\end{pro}
The next lemma will be useful to prove propositions \ref{lemm:5} and \ref{lemm:6}.

\begin{lemm}
\label{lemm:1}
For all $1 \le \mu < \mu' \le 2$ and $k>0$, there is a constant $C$ such that
$$ 
\avert{\mu Q}|u|^2 \le \frac {C} { (1+ R^2\avert{Q}V)^k} \big(\avert{\mu' Q} |u|^2\big).
$$
and
$$ 
\avert{\mu Q} (|L u|^2 + V|u|^2) \le \frac {C} { (1+ R^2\avert{Q}V)^k} \big(\avert{\mu' Q} (|L u|^2 + V|u|^2) \big).
$$
\end{lemm}
\begin{proof} There is nothing to prove if  $R^{2} \avert{Q} V \leq 1$. We assume $R^{2} \avert{Q} V > 1$. The well-known Caccioppoli type argument yields for $1\le \mu <\mu' \le 2$
\begin{equation}
\label{eq:caccio}
\int_{\mu Q} |L u|^2 + V|u|^2 \le \frac {C} {R^2} \int_{\mu' Q} |u|^2.
\end{equation}
The improved Fefferman-Phong inequality \eqref{1eq:FP} and the fact that the averages
of $V$ on $\mu Q$ with $1\le \mu \le 2$ are all uniformly comparable imply for some $\beta>0$,   
$$
 \frac{ 1}{ R^{2}}  \int_{\mu Q} |u|^2 \le \frac C {(R^2 \avert{Q}V)^\beta}\int_{\mu Q}  |L u|^2 + V|u|^2  .
$$ 
The desired estimates follow readily by iterating these two inequalities.
\end{proof}

\begin{lemm}
\label{lemm:2}
For all $1<\mu \le 2$ and $k>0$, there is a constant $C$ such that  
$$
(R\avert{Q}V)^2 \, \avert{Q} |u|^2 \le \frac {C} { (1+ R^2\avert{Q}V)^k}  \big(\avert{\mu Q}   V|u|^2 \big).
$$ 
\end{lemm}
\begin{proof} Using Lemma \ref{lemm:1} with $k>1$ and $1<\mu' <\mu$ and subsequently Lemma \ref{2th:sousharm}, we have:
$$
(R \avert{Q} V)^2 \avert{ Q}|u|^2 \le \frac {C  \avert{Q}V\,\avert{ \mu' Q}|u|^2} { (1+ R^2\avert{Q}V)^{k-1}}  \le \frac {C  \avert{\mu' Q}V\,\sup_{ \mu' Q}|u|^2} { (1+ R^2\avert{Q}V)^{k-1}}   \le \frac {C  \avert{\mu Q} (V|u|^2)} { (1+ R^2\avert{Q}V)^{k-1}}  .
$$
 
 \end{proof}

\begin{lemm}
\label{lemm:3}
For all $1<\mu\le 2$, $k>0$ and $n <p<\infty$, there is a constant $C$ such that
$$
(R\avert{Q}V)^2 \, \avert{Q} |u|^2 \le  \frac {C} { (1+ R^2\avert{Q}V)^k}  \big( \avert{\mu Q}   |L u|^p \big)^{2/p}.
$$
\end{lemm}
\begin{proof} If $\avert{\mu Q} |L u|^p
=\infty$ , there is nothing to prove. Assume, therefore, that $\avert{\mu Q} |L u|^p
<\infty$.
  Let $1<\nu <\mu$ and $\eta$ be a smooth non-negative function, bounded by 1, equal to 1 on $\nu Q$ with support on $\mu Q$ and whose gradient is bounded by $ C/R$ and Laplacian by $C/R^2$.

Integrating the equation $H(\textbf{a},0) u + Vu =0$ against� $\bar{u}\eta^2$. \\
Since

 $$H(\textbf{a},V)u = \sum_{j=1}^{n} L_{j}^{\star} L_{j} u + V u,$$
  
$$\int H(\textbf{a},V)u\,\bar{u}\eta^2 =  \sum_{j=1}^{n}\int  L_{j} u\,\overline{L_{j}(u\eta^2)} + \int V |u|^{2}\,\eta^2,$$
then
$$
\int  |L u|^2 \eta^2 + V|u|^2 \eta^2= 2 \int L u \cdot \nabla \eta\, \bar{u} \eta, $$
hence
$$\int  V|u|^2 \eta^2\le \frac C R \bigg( \int_{\mu Q} |L u|^2 \bigg)^{1/2} \bigg( \int |u|^2 \eta^2\bigg)^{1/2},$$
\begin{equation}\label{X}
X \le  {C\, {(R^2\avert{Q}V)^{1/2}}  |\mu Q|^{1/2}\, Y^{1/2}\, Z^{1/2}}
\end{equation}
where we set $X=(R^2\avert{Q}V) \int V |u|^2 \eta^2$,  $Y= \big(\avert{\mu Q} |L u|^p\big)^{2/p}$
and $Z= \avert{Q}V \int |u|^2 \eta^2$.
By Morrey's embedding theorem and diamagnetic inequality \eqref{2diam}, $u$ is H\"older continuous with exponent $\alpha= 1- n/p$. Hence for all $x,y\in \mu Q$, we have

$$
\left||u(x)| -|u(y)|\right| \le C \bigg( \frac{|x-y|}R\bigg)^\alpha \, R\, \big(\avert{\mu Q} |\nabla |u||^p\big)^{1/p}\leq C \bigg( \frac{|x-y|}R\bigg)^\alpha \, R\, Y^{1/2}.
$$
 We pick $y\in \overline Q$ such that $|u(y)|= \inf_{Q}|u|$. Then 
\begin{align*}
Z=\avert{Q}V \int |u|^2 \eta^2 &\le 2  (\avert{Q} V) \inf_{Q}|u|^2 \int \eta^2 + 2 (\avert{Q} V) \int \left||u(x)| -|u(y)|\right|^2 \eta^2(x)\, dx
\\
& \le 2 \big(\avert{Q} (V|u|^2)\big) \int \eta^2 + C (\avert{Q} V) R^2 Y \, \int 
 \bigg( \frac{|x-y|}R\bigg)^{2\alpha} \eta^2(x)\, dx
 \\
 & \le C   \big(\avert{Q} (V|u|^2)\big)  |Q| + C (\avert{Q} V) R^2 Y\,|\mu Q|  \\
 & \le C  \int V|u|^2\eta^2 + C (\avert{Q} V) R^2 Y\,|\mu Q|.
 \end{align*}
 where, in the penultimate inequality, we used the support condition on $\eta$ and $0\le \eta \le 1$, and in the last,    $\eta=1$ on $ Q$. 
 Using the previous inequalities, we obtain
 $$X  \le C  |\mu Q| ^{1/2} \, Y^{1/2} \, \big( CX + C(R^2\avert{Q}V)^2  |\mu Q| Y\big)^{1/2},
 $$
 which, as $2ab \le \epsilon^{-1} a^2 + \epsilon b^2$ for all $a,b \ge 0$ and $\epsilon>0$, implies
  $$
 X \le C(1 + R^2 \avert{Q}V )^2 \,  |\mu Q| \, Y.
 $$
 Next, let $1< \nu'<\nu$. Using $\eta=1$ on $\nu Q$ Lemma \ref{2th:sousharm} and Lemma \ref{lemm:1}
 $$
 \int  V|u|^2 \eta^2 \ge \int_{\nu Q} V |u|^2 \ge C \avert{\nu' Q } V \, \int_{\nu' Q} |u|^2 \ge C   (\avert{Q } V)  (1 + R^2 \avert{Q}V )^k \int_{ Q} |u|^2,
 $$
 hence
 $$
 X \ge C (R \avert{Q}V)^2  (1 + R^2 \avert{Q}V )^k \int_{Q} |u|^2.
 $$
 The upper and lower bounds for $X$ yield the lemma.

 \end{proof}
 
 \begin{lemm}\label{th:Lu}Let $q<n$, there exists a constant $C>0$ such that 
  \begin{equation}\label{eq1'}(\avert{Q}|Lu|^{q^{\star}})^{1/q^{\star}}\leq C (\frac{1}{R}+R\avert{Q} V)  (\avert{3Q} | u|^{2})^{1/2} .
   \end{equation}
   Consider $q\geq n$, there is a constant $C>0$ such that    \begin{equation}\label{eq1''}\sup_{Q}|Lu|\leq C (\frac{1}{R}+R\avert{Q} V)  (\avert{3Q} | u|^{2})^{1/2} .
   \end{equation}
 \end{lemm}
 \begin{proof} Set $\phi \in C_{0}^{\infty}(2Q)$, with $\phi \equiv 1$ in $Q$ , $|\nabla \phi|\leq {C}/{R}$ and $|\nabla^{2} \phi|\leq {C}/{R^{2}}$.\\ Since
 $$ H(\textbf{a},0)(u \phi)= \frac{2}{i} Lu. \nabla \phi- u \Delta \phi  - V u \phi,$$
then
 $$ u(x)\phi(x) = \int_{\mathbb{R}^n} \Gamma_{0}(x,y)[\frac{2}{i} Lu(y) .\nabla \phi(y)- u(y) \Delta \phi(y)  - V(y) u(y) \phi(y)]\, dy. $$
By \eqref{FS}, we obtain for all $x_{0} \in Q$
 
 $$|Lu(x_0)|\leq \frac{C}{R^n}\int_{2Q} |L u(y)|dy + \frac{C}{R^{n+1}}\int_{2Q} | u(y)|dy +C \int_{2Q} \frac{ V(y) |u(y)|}{|x_{0}-y|^{n-1}}dy. $$
 Using Caccioppoli type inequality, it follows that  
 $$|Lu(x_0)|\leq  \frac{C}{R}(\avert{2Q} | u(y)|^{2}dy)^{1/2} + C\int_{2Q} \frac{ V(y) |u(y)|}{|x_{0}-y|^{n-1}}dy. $$
  If $q< n$,
   $$(\avert{Q}|Lu|^{q^{\star}})^{1/q^{\star}}\leq  \frac{C}{R} \sup_{ 2Q} | u| + C\bigg{(}\avert{2Q} \left\{\int_{2Q}\frac{ V(y) |u(y)|}{|x_{0}-y|^{n-1}}dy\right\}^{q^{\star}}dx\bigg{)}^{\frac{1}{q^{\star}}} .$$
  By Hardy-Littlewood-Sobolev inequality, we obtain
 \begin{equation}\label{eq:1'}(\avert{Q}|Lu|^{q^{\star}})^{1/q^{\star}}\leq  \frac{C}{R} \sup_{\frac{5}{2} Q} | u| + C\, R(\avert{2Q} |Vu|^{q})^{1/q} 
 \end{equation}
  $$\qquad\leq  \frac{C}{R} \sup_{\frac{5}{2}Q} | u| + C\,R(\avert{Q} |V|^{q})^{1/q} \sup_{2Q} |u|$$
   $$\qquad\leq  \frac{C}{R} \sup_{ \frac{5}{2}Q} | u| +C\,R\avert{Q} |V| \sup_{\frac{5}{2}Q} |u|.$$
  Subharmonicity of $|u|^{2}$ yields 
   $$(\avert{Q}|Lu|^{q^{\star}})^{1/q^{\star}}\leq C (\frac{1}{R}+R\avert{Q} V)  (\avert{3Q} | u|^{2})^{1/2}.$$
   If $q\geq n$
   $$\sup_{Q}|Lu|\leq  \frac{C}{R} \sup_{ 2Q} | u| + C \sup_{2Q} | u(y)| \sup_{x\in Q} \bigg{(}\int_{2Q}\frac{ V(y) }{|x-y|^{n-1}}dy\bigg{)} $$
   $$\leq  \frac{C}{R} \sup_{ 2Q} | u| +  \frac{C}{R^{n-1}}  \sup_{2Q} | u| \int_{2Q} V(y)dy  .$$
Here we used H\"older inequality  with $V\in L^q(2Q)$ and the fact that $V\in RH_q$.
Hence, inequality \eqref{eq1''} holds. 
\end{proof}

\begin{lemm}
\label{lemm:4}
 Let $1< \mu\le 2$ and $k>0$, if $n/2 \le q <n$, then there is a constant $C$ such that  
$$
\big( \avert{ Q}   |L u|^{ q^*} \big)^{1/ q^*} \le  \frac {C} { R(1+ R^2\avert{Q}V)^k} \big( \sup_{\mu Q} |u|\big). 
$$
If  $q \ge n$, then there is a constant $C$ such that 
$$
 \sup_{ Q}   |L u|  \le    \frac {C} { R(1+ R^2\avert{Q}V)^k}  \big(\sup_{\mu Q} |u|\big). 
$$

\end{lemm}
\begin{proof}It suffices to combine Lemma \ref{th:Lu} with Lemma \ref{lemm:1}. 
\end{proof}

\subsubsection{Proof of Proposition \ref{lemm:5}  }

\begin{proof}
 We assume $q>\frac{2n}{n+2}$.
 
Let $v$ be a weak solution of $H(\textbf{a},0)v=0$ in $ 2 Q$ with $v=u$ on $\partial( 2 Q)$ and  set $w=u-v$ on $ 2 Q$. Since $w=0$ on $\partial( 2 Q)$, we have 
$$
 (\avert{ 2 Q}   |L w|^2 \big)^{1/2}  \le 
(\avert{ 2 Q}   |L u|^2 \big)^{1/2} .
$$
By estimates of Lemma \ref{th:RH1}, we have for all $2\le p \le \infty$ and in particular for $p=\tilde q$,  
$$
\big ( \avert{Q} |L v|^p \big)^{1/p} \le C (\avert{ \frac{3}{2} Q}   |L v|^2 +\avert{ \frac{3}{2} Q} |m(.,|B|)\,v|^{2} \big)^{1/2}.$$ 
The subharmonicity of $|v|^{2}$ and $|u|^{2}$ implies 
$$\avert{\frac{3}{2}Q}|v|^{2}\leq \sup_{2Q} |v|^{2}=\sup_{\partial(2Q)}|v|^{2}=\sup_{\partial (2Q)}|u|^{2}\leq C\avert{3Q} |u|^{2}.$$
Hence
 \begin{align*}\big{(} \avert{\frac{3}{2} Q}| m(x,|B|) v(x) |^{2}dx \big{)}^{1/2}\leq & \{1+ R m(x_{0},|B|)\}^{k_{0} }m(x_{0},|B|)\big{(} \avert{\frac{3}{2}Q}|  v|^{2} \big{)}^{1/2}
\\
&\leq \frac{C_{k} \{1+ R m(x_{0},|B|)\}^{k_{0}}m(x_{0},|B|)}{\{1+ R m(x_{0},|B|)\}^{k}}\big{(} \avert{3Q}|u |^{2} \big{)}^{1/2}
\\
&\leq \{1+ R m(x_{0},|B|)\}^{k_{0}-k+(k_{0}/k_{0}+1)}\frac{C_{k} m(x_{0},|B|)}{\{1+ R m(x_{0},|B|)\}^{k_{0}/k_{0}+1}}\big{(} \avert{3Q}|  u |^{2} \big{)}^{1/2}
\\
 &\leq  C\big{(} \avert{3Q}|m(.,|B|)  u |^{2} \big{)}^{1/2}.
\end{align*}
 Where we used Lemma \ref{th:mprop} and Lemma \ref{lemm:1} for an arbitrary $k$.
It follows $$
\big ( \avert{Q} |L v|^p \big)^{1/p} \ \le C (\avert{ 3 Q}   |L u|^2 +\avert{3Q} |m(.,|B|)\,u|^{2} \big)^{1/2} .$$ 

Let $1<\mu <2$ and  $\eta$ be a smooth non-negative function, bounded by 1, equal to 1 on $ Q$ with support contained in $\mu Q$ and whose gradient is bounded by $ C/R$ and Laplacian by $C/R^2$.
As $H(\textbf{a},0) w= H(\textbf{a},0) u = - V u $ on $2Q$, we have
 $$ H(\textbf{a},0)(w \eta)= \frac{2}{i} Lw .\nabla \eta- w \Delta \eta  - V u \eta.$$
Hence
$$
 L(w\eta)(x) = \int_{\mathbb{R}^n} L^{x}\Gamma_{0}(x,y) \big[\frac{2}{i} L(w)(y) .\nabla \eta(y)- w(y) \Delta \eta(y)  - (V u \eta)(y)\big]\, dy$$$$= I+II+III,
$$
with $\Gamma_{0}$ the kernel of $H(\textbf{a},0)^{-1}$. We know by \eqref{FS}, $|L^{x}\Gamma_{0}(x,y) | \le C |x-y|^{1-n}$.

Since $\tilde q \le q^*$, then
$$
\big( \avert{Q} |L w|^{\tilde q} \big)^{1/\tilde q } \le \big( \avert{Q} |L w|^{ q^*} \big)^{1/ q^* }.
$$

Using support conditions on $\eta$, we obtain the following estimates for all $x\in Q$,
$$
|I|\le  C \big( \avert{ 2 Q} |L w|^2 \big)^{1/2} \le C \big( \avert{ 2 Q} |L u|^2 \big)^{1/2}
$$
and
$$
|II| \le \frac C R \avert{ 2 Q} | w| \le C \big( \avert{ 2 Q} |\nabla |w||^2 \big)^{1/2} \le C \big( \avert{ 2 Q} |Lw|^2 \big)^{1/2} \le C \big( \avert{ 2 Q} |L u|^2 \big)^{1/2},
 $$
 Above we used the Poincar\'e and the diamagnetic inequality \eqref{2diam} \footnote{We consider the function $\tilde{w}$ defined as $\left\{\begin{array}{ll}
\tilde{ w}=w,\,\, \textrm{sur}\,\, 2Q\\\tilde{w}=0, \,\, \textrm{sur}\,\, \mathbb{R}^{n}\setminus{2Q}
  \end{array}
  \right.$. Then $L(\tilde{w})=\mathbf{1}_{2Q} L(w)$ as $w$ vanishes on  $\partial 2Q$. }
 
 It follows by Hardy-Littlewood-Sobolev inequality,
$$
\bigg(\int_{\mathbb{R}^n} III^{q^*}\bigg)^{1/q^*} \le C\bigg( \int_{
\mathbb{R}^n} |Vu\eta|^q \bigg)^{1/q}  \le C\bigg( \int_{
\mu Q} |V|^q \bigg)^{1/q} \sup_{\mu Q } |u|.
$$
Since $V\in RH_{q}$, then
\begin{equation}
\label{eq:I}
\big (\avert{Q} III^{q^*}\big)^{1/q^*} \le C R \, \avert{\mu Q} V \sup_{\mu Q } |u|.
\end{equation}
Now, if $\mu < \mu' < 2$, subharmonicity of $|u|^{2}$ and Lemma \ref{2th:sousharm} yield
$$
R\, \avert{\mu Q} V \sup_{\mu Q}|u| \le  C R\,  \avert{\mu' Q} V  \,  \big(\avert{\mu' Q} |u|^2 \big)^{1/2},
$$
 which by Lemma \ref{lemm:2} is bounded by 
 $
C  \big(  \avert{2Q} (V|u|^2)\big)^{1/2}
$.
Gathering the estimates obtained for $L v$ and $L w$, the lemma is proved.
\end{proof}

\subsubsection{Proof of Proposition \ref{lemm:6}  }

\begin{proof}  Assume $q>n/2$ (it includes $q=\frac{n}{2}$ via the self-improvement of reverse H\"older classes). The previous lemma shows that $ \avert{\mu' Q} |L u|^{ \tilde q}  <\infty$ for all $1<\mu'\le \mu$. As $\tilde q =2q>n$,  Lemma \ref{lemm:3} applies and using it for $k=0$ instead of Lemma \ref{lemm:2} in the argument of Lemma  \ref{lemm:5}, we obtain,  
$$
\big( \avert{Q} |L w|^{ q^*} \big)^{1/ q^* } \le C \big( \avert{\mu Q} |L u|^{ 2q} \big)^{1/ 2q }.
$$
Next, we know that $$
\big( \avert{Q} |L v|^{ q^*} \big)^{1/ q^* } \le C \big( \avert{\mu Q} |L u|^{ 2q}+|m(.,|B|)u|^{2q} \big)^{1/ 2q }.$$
Hence 
$$
\big( \avert{Q} |L u|^{ q^*} \big)^{1/ q^* } \le C \big( \avert{\mu Q} |L u|^{ 2q}+|m(.,|B|)u|^{2q} \big)^{1/ 2q }.
$$
\end{proof}

\subsection{Maximal inequalities}
 \textit{Proof of Theorem \ref{2th:VH}:}\\
 	
The proof of this theorem is identical to that of Theorem 1.1 in \cite{AB}. First we prove an $ L ^ 1$ inequality, then we establish some reverse H\"older type estimates, then finally we apply Theorem \ref{2theor:shen}.

\begin{lemm}\label{lemm:l1}Let $f \in L^{\infty}_{comp}(\mathbb{R}^n)$ and $u=H(\textbf{a},V)^{-1}f$. Then, 
\begin{equation}\label{2Vu}
\int_{\mathbb{R}^n} V|u| \le \int_{\mathbb{R}^n} |f|, 
\end{equation}
and
\begin{equation}\label{2Hu}
\int_{\mathbb{R}^n} |H(\textbf{a},0)u| \le 2 \int_{\mathbb{R}^n} |f|.
\end{equation}
 
\end{lemm}

\begin{proof}
 $V\geq 0$, by Kato-Simon inequality \eqref{eq:KS}, we have
 $$|H(\textbf{a},V)^{-1}f|\leq  H(0,V)^{-1}|f|.$$
 We know, by \cite{AB} that
 $$\int_{\mathbb{R}^n} VH(0,V)^{-1}|f| \le \int_{\mathbb{R}^n} |f|.$$ Thus, inequality
 \eqref{2Vu} holds, and inequality \eqref{2Hu} follows by difference.
\end{proof}

\textit{ Proof of the $L^p$ maximal inequality:} 
Assume $V\in RH_q$ with $q> 1$.	
$V H (\textbf{a}, V) ^{-1} $. We know that this operator is bounded on $ L ^ 1 (\mathbb{R}^ n) $, so we apply Theorem \ref{2theor:shen} through the reverse H\"older inequality verified by any weak solution.
Set $Q$ a fixed cube and $f\in L^\infty(\mathbb{R}^n)$ a function with compact support in $\mathbb{R}^n \setminus 4Q$. Then $u=H(\textbf{a},V)^{-1} f$ is well defined in  $ \dot{\mathcal{V}}$ and it is a weak solution of $H(\textbf{a},0) u + Vu=0$ in $4Q$.

Since $|u|^2$ is subharmonic, by Lemma \ref{2th:sousharm} with $w=V$, $f=|u|^2$ and $s=1/2$, we obtain
$$\big{(}\avert{Q} |Vu|^{q}\big{)}^{1/q} \leq C \avert{2Q} |Vu|.$$
Thus \eqref{2T:shen} holds with $T=V H(\textbf{a},V)^{-1}$,  $p_{0}=1$, 
$q_{0}=q$, $S=0$, $\alpha_{1}=2$ and $\alpha_{2}=4$. Hence $V H(\textbf{a},V)^{-1}$ is bounded on $L^p(\mathbb{R}^n)$ for all $1<p<q$ by Theorem \ref{2theor:shen}. Due to the properties of $RH_{q}$ weights, we can replace $q$ by $q+\epsilon$. Taking the difference, we obtain the same result for $H(\textbf{a},0) H(\textbf{a},V)^{-1}$.
This completes the proof of Theorem \ref{2th:VH} .$\square$
 \begin{rem} Theorem \ref{Maxbes} is a consequence of Theorem \ref{Maxshen} and \ref{2th:VH}: $$L_{s}L_{k}H(\text{a},V)^{-1}=L_{s}L_{k}H(\text{a},0)^{-1}H(\text{a},0)H(\text{a},V)^{-1}.$$
 \end{rem}

     \subsection{Proof of Theorem \ref{th:2a}}
	
Using Theorem \ref{th:1a} and the corollary \ref{2th:VH2}, we can establish a first result:
 \begin{theo}\label{th:Lm}Under the assumptions of Theorem \ref{th:2a}, there exists an $ \epsilon >0$ such that $L\, H(\textbf{a},V)^{-1/2}$ is $L^p$ bounded for all $1 <p
< 2q+\epsilon$, where $\epsilon$  depends only on $V$.
\end{theo}
\begin{proof} ? $$LH(\textbf{a},V)^{-1/2}=LH(\textbf{a},0)^{-1/2}H(\textbf{a},0)^{1/2}H(\textbf{a},V)^{-1/2}.$$ 
\end{proof} 
\begin{rem}\label{m2q}Using the same argument, we obtain that $m(.,|B|)\,H(\textbf{a},V)^{-1/2}$ is $L^p$ bounded for $1\leq p<2q+\epsilon$. 
\end{rem}
  	
Now, we have to controll the term $ m (. | B |) u $ appearing in the previous estimates. It suffices to study the $L^p$ boundedness of operator $m(.,|B|)\,H(\textbf{a},V)^{-1/2}$.	
The result of the remark \ref{m2q} is not enough, we will improve it through the following theorem:
   \begin{theo}\label{th:mmh}Let $\textbf{a}\in L^{2}_{loc}(\mathbb{R}^{n})^{n}$, $V\in RH_{q}$, $1<q\leq +\infty$ and we assume \eqref{eq:shen}.
  Then, for all $1\leq p \leq \infty$, there is a constant $C_p$, such that 
 \begin{equation}\| m(.,|B|)^{2}H(\textbf{a},V)^{-1}(f)\|_{p}\leq C \| f\|_{p}
 \end{equation}
 for all $f\in C_{0}^{\infty}(\mathbb{R}^{n}).$

 \end{theo}
By complex interpolation, we obtain
 \begin{cor}\label{cormu} Suppose $\textbf{a}\in L^{2}_{loc}(\mathbb{R}^{n})^{n}$ and $V\in RH_{q}$, $1<q\leq +\infty$. We also assume \eqref{eq:shen}.
  Then, for all $1\leq p < \infty$, there is a constant $C_p$, such that 
 \begin{equation}\| m(.,|B|)\,H(\textbf{a},V)^{-1/2}(f)\|_{p}\leq C \| f\|_{p},
 \end{equation}
 for all $f\in C_{0}^{\infty}(\mathbb{R}^{n}).$

 \end{cor}
 We will apply Theorem \ref{2theor:shen} to prove Theorem \ref{th:mmh} for $p>2$ and we will need the following lemma:
 \begin{lemm}\label{lemm:solV} Under assumptions of Theorem \ref{th:mmh}, let $u$ be a weak solution of $H(\textbf{a},V)u=0$ in $4Q$ centered at $x_{0} \in \mathbb{R}^{n}$ and of sidelength $4R$. Then, for any integer $k>0$, there exists a constant $C_{k}$ such that 
 \begin{equation}\label{eq:solV}|u(x_{0})| \leq \frac{C_{k}}{\{1+
Rm(x_{0},|B|)\}^{k} }\big(\avert{3Q} |u|^{2}\big)^{1/2}. 
\end{equation} 
  
  \end{lemm}
  \begin{proof}
	
We will use the results obtained in the absence of electric potential $ V $. For $f\in C^{\infty}_{0}(\mathbb{R}^n)$,
  \begin{equation}\label{eq:mL}\|m(.,|B|)f\|_{2} \leq C\|H(\textbf{a},0)^{1/2}f\|_{2} \leq C\|Lf\|_{2}.
  \end{equation}
 Consider $\phi$ a smooth non-negative function, bounded by 1, equal to 1 on $ Q$ with support in $\frac{3}{2}\, Q$ and whose gradient is bounded by $ C/R$.
  
We apply inequality \eqref{eq:mL} to� $u\phi$ and we obtain
  $$\int_{\mathbb{R}^{n}} | m(.,|B|)\, u\phi|^{2} \leq C \int_{\mathbb{R}^{n}} |L(u\phi)|^{2}.$$
  This gives $$\int_{Q} | m(.,|B|)\, u|^{2} \leq C \int_{\frac{3}{2} \,Q} |\phi Lu|^{2}+ \int_{\frac{3}{2} \,Q} |u\nabla \phi|^{2}$$
  $$\int_{Q} | m(.,|B|)\, u|^{2} \leq C \int_{\frac{3}{2} \,Q} |L\,u|^{2}+ \frac{C}{R^{2}}\int_{\frac{3}{2} \,Q} |u|^{2} \leq \frac{C}{R^{2}}\int_{2Q} |u|^{2},$$
  where we used Caccioppoli type inequality.
 Now, Lemma \ref{th:mprop} yields
$$\int_{Q} |u|^{2}\leq \frac{C\{1+
Rm(x_{0},|B|)\}^{2k_{0}/(k_{0}+1)}}{\{Rm(x_{0},|B|)\}^{2} }\int_{3Q} |u|^{2}\leq \frac{C}{\{1+Rm(x_{0},|B|)\}^{2/(k_{0}+1)} }\int_{3Q} |u|^{2},$$
then
$$|u(x_{0})|\leq C\big(\avert{Q} |u|^{2}\big)^{1/2}  \leq \frac{C_k}{\{1+
Rm(x_{0},|B|)\}^{k/(k_{0}+1)}} \big(\avert{3Q} |u|^{2}\big)^{1/2} .$$
\end{proof}
 \begin{pro}\label{pro:muv} Under assumptions of Theorem \ref{th:mmh}, let $u$ be a weak solution of $H(\textbf{a},V)u=0$ in $4Q$, for all $s>2$, there exists a constant $C>0$ such that 
 \begin{equation}\label{muvv}
 \big{(} \avert{Q}| m(.,|B|)^{2} u |^{s} \big{)}^{1/s}\leq C \big{(} \avert{3Q}| m(.,|B|)^{2} u |^{2} \big{)}^{1/2}.
 \end{equation}

 \end{pro}
the proof is similar to that of Proposition \ref{pro:mu}.
\\
\\
  \textit{Proof of Theorem \ref{th:mmh}:}\\
 We have  $$m(.,|B|)^{2}H(\textbf{a},V)^{-1}=m(.,|B|)^{2}H(\textbf{a},0)^{-1}H(\textbf{a},0)H(\textbf{a},V)^{-1}.$$
  It follows by Theorem \ref{2th:VH} that $H(\textbf{a},0)H(\textbf{a},V)^{-1}$ is $L^p$ bounded for $1\leq p<q+\epsilon$. We know also that $m(.,|B|)^{2}H(\textbf{a},0)^{-1}$ is $L^p$ bounded for $1<p<\infty$. Hence $m(.,|B|)^{2}H(\textbf{a},V)^{-1}$ is bounded on $L^{p}(\mathbb{R}^{n})$ for all $1<p<q+\epsilon$. In particular it is  $L^{2}$ bounded.
Then we apply Theorem \ref{2theor:shen} to study the behaviour of this operator on $ L ^{p} (\mathbb{R} ^{n}) $.   
Fix a cube $Q$ and let $f\in C^{\infty}_{0}(\mathbb{R}^n,\mathbb{C})$ compact support contained in $\mathbb{R}^{n} \setminus 4Q$. Then $u=H(\textbf{a},V)^{-1}f$ is well defined on $\mathbb{R}^n$. Due to the support conditions on $f$, $u$ is a weak solution of $H(\textbf{a},V)u=0$ on $4Q$. It follows by Proposition \ref{pro:muv} that, for all $s>2$, there is a constant $C$, independant of $Q$ and $\textbf{F}$, such that 
 \begin{equation}
  \big{(} \avert{Q}| m(.,|B|)^{2}H(\textbf{a},V)^{-1}f |^{s} \big{)}^{1/s}\leq C \big{(} \avert{3Q}| m(.,|B|)^{2} H(\textbf{a},V)^{-1}f|^{2}\big{)}^{1/2}.
  \end{equation}
  Then \eqref{2T:shen} holds with $T=m(.,|B|)^{2}H(\textbf{a},V)^{-1},\, q_{0}=s,\, p_{0}=2$ and $T$ is $L^p$ bounded by Theorem \ref{2theor:shen}.$\square$

 \begin{rem}Note that we can prove Corollary \ref{cormu} by a proof analogous to that of Theorem \ref{th:mmh}. In fact, under hypothesies of Corollary \ref{cormu}, if $u$ is a weak solution of $H(\textbf{a},V)u=0$ in the cube $4Q$ centred at $x_{0} \in \mathbb{R}^{n}$ of sedelength $4R$. Then, for all $s>2$, there exists a constant $C>0$ such that
 
  \begin{equation}
 \big{(} \avert{Q}| m(.,|B|) u |^{s} \big{)}^{1/s}\leq C \big{(} \avert{3Q}| m(.,|B|)^{2} u |^{2} \big{)}^{1/2}.
 \end{equation}
 \end{rem}

  \space
  \space

 \textit{Proof of Theorem \ref{th:2a}: }\\
  	
We know that for $ p\leq $ 2 and without conditions on $ V $ operators $ LH (\textbf{a}, V) ^ {-1 / 2} $ and $ V ^ {1 / 2} H (\textbf{a}, V) ^ {-1 / 2}$ are $ L ^ p$ bounded .	
We would therefore limit ourselves to cases where $p>2$.

 The following lemma allows the reduction of the problem.
 \begin{lemm} Under the assumptions of Theorem \ref{th:2a},  $L H(\textbf{a},V)^{-1/2}$ is $L^p$ bounded if and only if $L H(\textbf{a},V)^{-1} L^{\star}$ and $L H(\textbf{a},V)^{-1}\,V^{1/2}$ are $L^{p}$ bounded.
 \end{lemm}
  The proof of this lemma is similar to that of Lemma \ref{red}. 	
	
We also use the following results:
 \begin{pro} \label{2thC''} Assume $V \in RH_q$ with $1<q\leq \infty$, then there is an $\epsilon >0$ such that for all $p$ with $2<p< 2(q+\epsilon)$, there exists a constant $C_p$ depending on $V$, such that $f\in C^\infty_{0}(\mathbb{R}^n,\mathbb{C})$ and $\textbf{F}\in C_{0}^\infty(\mathbb{R}^n, \mathbb{C}^n)$, 
$$
\|V^{1/2} H(\textbf{a},V)^{-1} V^{1/2}f\|_p  \le C_{p} \|f\|_{p}, \quad  \|V^{1/2} H(\textbf{a},V)^{-1} L^{\star} \textbf{F}\|_p \le C_p\|\textbf{F}\|_p.
$$
\end{pro}
\begin{proof}

%

 Fix a cube $Q$ in $\mathbb{R}^{n}$ and let $f\in C^\infty_{0}(\mathbb{R}^n)$ supported away from $4Q$. Then $u=H(\textbf{a},V)^{-1}V^{1/2}f$ is well-defined on $\mathbb{R}^n$ with
$\|V^{1/2} u\|_{2}+ \|L u \|_{2} \le \|f\|_{2}$, by construction of $H(\textbf{a},V)$ and
$$
\int_{\mathbb{R}^n} V u \varphi + \nabla u \cdot \nabla \varphi = \int_{\mathbb{R}^n} V^{1/2} f \varphi
$$ for all $\varphi\in L^2$ with $\|V^{1/2} \varphi\|_{2}+ \|\nabla \varphi \|_{2}<\infty$.
In particular, the support condition on $f$ implies that $u$ is a weak solution of $H(\textbf{a},V)u=0$ in $4Q$, hence $|u|^2$ is subharmonic on $4Q$. Consider $r$ such that $V\in RH_{r}$ and note that $V^{1/2} \in RH_{2r}$. By Lemma  \ref{2th:sousharm} with $\omega=V^{1/2}$ $f=|u|^2$ and  $s=1/2$, we have
$$
\big(\avert {Q} (V^{1/2}|u| )^{2r} \big)^{1/2r} \le C \, \avert {\mu Q} (V^{1/2} |u|).
$$
Hence \eqref{2T:shen} holds with $T=V^{1/2}H(\textbf{a},V)^{-1} V^{1/2}$, $q_{0}=2r$, $p_{0}=2$ and $S=0$. By Theorem \ref{2theor:shen}, 
$V^{1/2}H^{-1} V^{1/2}$ is then $L^p$ bounded for $2<p<2r$. 

We use the same argument to obtain that $V^{1/2}H(\textbf{a},V)^{-1}\,L^{*}$ is $L^p$ bounded for $2<p<2r$.

\end{proof}
	
To prove Theorem \ref{th:2a}, it suffices to prove the following result:
\begin{pro} \label{thD''}
Assume $V \in RH_q$ with $q>1$. 
If $2<p<  q^{*}+\epsilon$ for an $\epsilon>0$ which depends on the $RH_{q}$ constant of $V$, then for all $f\in C^\infty_{0}(\mathbb{R}^n,\mathbb{C})$ and $\textbf{F}\in C_{0}^\infty(\mathbb{R}^n, \mathbb{C}^n)$, 
$$
\|L  H(\textbf{a},V)^{-1} V^{1/2}f\|_p  \le C_{p} \|f\|_{p}, \quad  \|L H(\textbf{a},V)^{-1} L^{\star} \textbf{F}\|_p \le C_p\|\textbf{F}\|_p.
$$
\end{pro}

\begin{proof}
Assume $q<n/2$. 
 Fix a cube $Q$ and let  $\textbf{F}\in C^{\infty}_{0}(\mathbb{R}^n,\mathbb{C}^n)$ supported away from $4Q$. Set $H=H(\textbf{a},V)$, $u=H^{-1} L^{\star}\textbf{F}$ is well-defined on $\mathbb{R}^n$. As before,  the support condition on  $\textbf{F}$, implies that $u$ is a weak solution of $Hu=0$ on $4Q$. Lemma \ref{lemm:5} implies for all $p\leq q^{\star}$
 \begin{equation}
  \big{(} \avert{Q}| LH^{-1}L^{\star} \textbf{F}|^{p}dx \big{)}^{1/p}\leq C \big{(} \avert{3Q}| L H^{-1} L^{\star}\textbf{F}|^{2}+| m(.,|B|)H^{-1} L^{\star}\textbf{F} |^{2}+|V^{1/2}H^{-1}L^{\star}\textbf{F}|^{2} \big{)}^{1/2}.
  \end{equation}
  Then \eqref{2T:shen} holds with $$T=LH^{-1}L^{\star},\, q_{0}=q^{\star},\, \,  p_{0}=2 \,\, \textrm{and}\,\, S \textbf{F}=\big(M\big(m(.,|B|)H^{-1} L^{\star}\textbf{F}+V^{1/2}H^{-1}L^{\star}\textbf{F}\textbf)^{2}\big)^{\frac{1}{2}},$$
  where $M$ is the maximal Hardy-Littlewood operator. Since S is $L^p$ bounded for all $1<p< 2q$  and $q^{\star}\leq 2q$, then $T$ is bounded on $L^p(\mathbf{R}^{n},\mathbf{C}^{n})$, $p< q^{\star}$. By the self-improvement of reverse H\"older estimates we can replace $q$ by a slightly larger value and, therefore, $L^p$ boundedness  for $p< q^*+\epsilon$ holds. $\footnote{Thanks to Theorem \ref{th:Lm}, we can improve the range of $p$: $1<p< 2q+\epsilon$.}$

Assume next that $ n/2 \le q < n$, then $q^{\star}\geq 2q$. We follow the same argument used for $p<n/2$, and we obtain first that $LH^{-1}L^{\star}$ is $L^p$ bounded for $q\leq 2q$. \\We can improve this result by Lemma \ref{lemm:6}: in fact, inequality \eqref{2T:shen} holds with $T=LH^{-1}L^{\star},\, q_{0}=q^{\star},\, p_{0}=2q \, \,\textrm{and}\, S=M\big(|m(.,|B|)H^{-1} L^{\star}|^{2}\big)^{\frac{1}{2}}$. Since $S$ is $L^p$ bounded for all $1<p<\infty$ then $T$ is bounded on $L^p(\mathbf{R}^{n},\mathbf{C}^{n})$, $p< q^{\star}$. Again, by self-improvement of the $RH_{q}$ condition, it holds for $p<q^{\star} +\epsilon$.
 
Finally, if $q\ge n$, then $L H^{-1} V^{1/2}$ is $L^p$ bounded for $2<p < \infty$. And this ends the proof.
\end{proof}

\end{document}